\documentclass{gtpartmod}
\pdfoutput=1
\usepackage{tikz}

\gtart

\title{Rigidity of spherical codes}

\author{Henry Cohn}
\givenname{Henry}
\surname{Cohn}
\address{Microsoft Research New England\\\newline
One Memorial Drive\\
Cambridge, MA 02142}
\email{cohn@microsoft.com}

\author{Yang Jiao}
\givenname{Yang}
\surname{Jiao}
\address{Physical Sciences--Oncology Center\\
Princeton University\\\newline
Princeton, New Jersey 08544}
\email{yjiao@princeton.edu}

\author{Abhinav Kumar}
\givenname{Abhinav}
\surname{Kumar}
\address{Department of Mathematics\\
Massachusetts Institute of Technology\\\newline
Cambridge, MA 02139}
\email{abhinav@math.mit.edu}

\author{Salvatore Torquato}
\givenname{Salvatore}
\surname{Torquato}
\address{Department of Chemistry\\
Princeton University\\\newline
Princeton, New Jersey 08544}
\email{torquato@electron.princeton.edu}

\volumenumber{15}
\issuenumber{4}
\publicationyear{2011}
\papernumber{55}
\startpage{2235}
\endpage{2273}

\doi{10.2140/gt.2011.15.2235}
\MR{}
\Zbl{}

\arxivreference{1102.5060}
\arxivpassword{}

\keyword{rigidity}
\keyword{jamming}
\keyword{packing}
\keyword{spherical codes}
\keyword{kissing problem}
\subject{primary}{msc2000}{52C25}
\subject{secondary}{msc2000}{52C17}

\received{24 February 2011}
\revised{23 May 2011}
\accepted{3 June 2011}
\published{23 November 2011}
\publishedonline{23 November 2011}
\proposed{Rob Kirby}
\seconded{Dmitri Burago, Joan Birman}
\corresponding{}
\editor{}
\version{}

\let\xysavmatrix\xymatrix
\def\xymatrix{\disablesubscriptcorrection\xysavmatrix}
\AtBeginDocument{\let\bar\wbar}

\newcommand{\sB}{\mathcal{B}}
\newcommand{\sC}{\mathcal{C}}
\newcommand{\sCl}[2]{\mathcal{C}_{#1\textup{--}#2}}
\newcommand{\F}{{\mathbb F}}
\newcommand{\Proj}{{\mathbb P}}

\newcommand{\eps}{\varepsilon}
\DeclareMathOperator*{\bigexpt}{\raisebox{-2.5pt}{\mbox{\LARGE $\mathbb{E}$}}}

\DeclareMathOperator*{\Prob}{\textrm{Pr}}
\renewcommand{\Re}{\mathop{\textrm{Re}}}

\theoremstyle{plain}
\newtheorem{theorem}{Theorem}[section]
\newtheorem{corollary}{Corollary}[section]
\makeatletter
 \let\c@corollary\c@theorem
\makeatother
\newtheorem{lemma}{Lemma}[section]
\makeatletter
 \let\c@lemma\c@theorem
\makeatother
\newtheorem{proposition}{Proposition}[section]
\makeatletter
 \let\c@proposition\c@theorem
\makeatother

\numberwithin{equation}{section}

\begin{document}

\begin{asciiabstract}
A packing of spherical caps on the surface of a sphere (that is, a
spherical code) is called rigid or jammed if it is isolated within the
space of packings.  In other words, aside from applying a global isometry,
the packing cannot be deformed.  In this paper, we systematically study
the rigidity of spherical codes, particularly kissing configurations.
One surprise is that the kissing configuration of the Coxeter-Todd
lattice is not jammed, despite being locally jammed (each individual
cap is held in place if its neighbors are fixed); in this respect, the
Coxeter-Todd lattice is analogous to the face-centered cubic lattice in
three dimensions. By contrast, we find that many other packings have
jammed kissing configurations, including the Barnes-Wall lattice and
all of the best kissing configurations known in four through twelve
dimensions. Jamming seems to become much less common for large kissing
configurations in higher dimensions, and in particular it fails for
the best kissing configurations known in 25 through 31 dimensions.
Motivated by this phenomenon, we find new kissing configurations in
these dimensions, which improve on the records set in 1982 by the
laminated lattices.
\end{asciiabstract}

\begin{abstract}
A packing of spherical caps on the surface of a sphere (that
is, a spherical code) is called rigid or jammed if it is
isolated within the space of packings.  In other words, aside
from applying a global isometry, the packing cannot be
deformed.  In this paper, we systematically study the rigidity
of spherical codes, particularly kissing configurations. One
surprise is that the kissing configuration of the Coxeter--Todd
lattice is not jammed, despite being locally jammed (each
individual cap is held in\break{} 
place if its neighbors are fixed); in this respect, the
Coxeter--Todd lattice is analogous\break{} 
to the face-centered cubic lattice in three dimensions. By
contrast, we find that many\break{} 
other packings have jammed kissing configurations, including
the Barnes--Wall\nobreak{} lat\nobreak{}tice\break{} 
and all of the best kissing configurations known in four
through twelve dimensions. Jamming seems to become much less
common for large kissing configurations in higher dimensions,
and in particular it fails for the best kissing configurations
known in $25$ through $31$ dimensions. Motivated by this
phenomenon, we find new kissing configurations in these
dimensions, which improve on the records set in 1982 by the
laminated lattices.
\end{abstract}

\maketitle

\section{Introduction}

One of the key qualitative properties of a packing is whether
it is jammed, that is, whether the particles are locked into
place.  Jamming is of obvious scientific importance if we are
using the packing to model a granular material.  Furthermore,
it plays a central role in studying local optimality of
packings, because one natural way to try to improve a packing
is to deform it so as to open up more space.

Jamming has been extensively studied for packings in Euclidean
space. See, for example, Torquato and Stillinger~\cite{TS} and
the references cited therein.  However, it has been less
thoroughly investigated in other geometries.  In this paper, we
investigate jamming for sphere packings in spherical geometry,
that is, packings of caps on the surface of a sphere.  Jamming
has previously been studied for sphere packings in $S^2$ (see
Tarnai and G\'asp\'ar~\cite{TG}), but there seems to have been
little investigation in higher dimensions.

A packing of congruent spherical caps on the unit sphere
$S^{n-1}$ in $\R^n$ yields a \emph{spherical code} (that is, a
finite subset of $S^{n-1}$) consisting of the centers of the
caps. The \emph{minimal distance} of such a code is the
smallest angular separation between distinct points in the
code.  In other words, the cosine of the minimal distance is
the greatest inner product between distinct points in the code.
The \emph{packing radius} is half the minimal distance, because
spherical caps of this radius centered at the points of the
code will not overlap, except tangentially.  A spherical code
is \emph{optimal} if its minimal distance is as large as
possible, given the dimension of the code and the number of
points it contains. (Note that this notion of optimality is
different from requiring that no more caps of the same size can
be added without causing overlap. Neither of these two notions
implies the other.)

Spherical codes arise naturally in many parts of mathematics
and science (see Cohn~\cite{C} for a more extensive discussion).
For example, in $\R^3$ they model pores in pollen grains or
colloidal particles adsorbing to the surface of a droplet in a
emulsion formed by two immiscible liquids. In higher
dimensions, they can be used as error-correcting codes for a
constant-power radio transmitter. Furthermore, many beautiful
spherical codes arise in Lie theory, discrete geometry, or the
study of the sporadic finite simple groups.

A \emph{deformation} of a spherical code is a continuous motion
of the points such that the minimal distance never drops below
its initial value. A deformation is an \emph{unjamming} if it
does not simply consist of applying global isometries (that is,
the pairwise distances do not all remain constant).  A
spherical code is called \emph{rigid} or \emph{jammed} if it
has no unjamming.  It is called \emph{locally jammed} if no
single point can be continuously moved while all the others are
held fixed.

For example, in the face-centered cubic packing of balls in
$\R^3$, the \emph{kissing configuration} (that is, points of
tangency on a given ball) consists of the vertices of a
cuboctahedron.  This code is locally jammed, but it is not in
fact jammed (see Conway and Sloane~\cite[page~29]{SPLAG} or
\fullref{prop:Anunjammed} below).  However, it can be deformed
into an optimal\break 
spherical code, namely the vertices of a regular icosahedron,
and the icosahedron is \break 
then a rigid code with a higher minimal distance than that of
the cuboctahedron.

As this example shows, deforming a spherical code is one way to
improve it.  Some optimal codes are not jammed; for example,
the best five-point codes in $S^2$ consist of two antipodal
points and three points orthogonal to them, and the three
points can move freely as long as they remain separated by at
least an angle of $\pi/2$.  Furthermore, computer experiments
suggest that an optimal code can have \emph{rattlers}, that is,
points not\pagebreak{} 
in contact with any other point, although no such case has ever
been rigorously analyzed. However, despite these issues,
rigidity is a powerful criterion for understanding when a\break 
code can be improved.

Note that whether a configuration is jammed depends on the
ambient space.  For example, the vertices of a square are
jammed in $S^1$ but not in $S^2$.

For infinite packings in Euclidean space, there are more subtle
distinctions between different types of jamming
(see Bezdek, Bezdek and Connelly~\cite{BBC98} and Torquato and
Stillinger~\cite{categories}) based on what sorts of motions are
allowed. For example, are all but finitely many particles held
fixed?  Are shearing motions allowed?  However, these issues do
not arise for packings in compact spaces.

Nevertheless, jamming seems to be a more subtle phenomenon on
spheres than it is in Euclidean space.  In Euclidean space,
there is an efficient algorithm to test for jamming (see Donev,
Torquato, Stillinger and Connelly~\cite{DTSC}) but on spheres
we do not know such an algorithm. The difficulty is caused by
curvature, which complicates certain arguments.  For example,
in Euclidean space every infinitesimal unjamming extends to an
actual unjamming, as we will explain in
\fullref{section:infinijam}, but the corresponding procedure
does not work on spheres.

\begin{table}[ht!]
\centering
\begin{tabular}{cc|cc}
Dimension & Kissing number & Dimension & Kissing number\\ \hline
$1$ & $\mathbf{2}$ & $17$ & $5346$\\
$2$ & $\mathbf{6}$ & $18$ & $7398$\\
$3$ & $\mathbf{12}$ & $19$ & $10668$\\
$4$ & $\mathbf{24}$ & $20$ & $17400$\\
$5$ & $40$ & $21$ & $27720$\\
$6$ & $72$ & $22$ & $49896$\\
$7$ & $126$ & $23$ & $93150$\\
$8$ & $\mathbf{240}$ & $24$ & $\mathbf{196560}$\\
$9$ & $306$ & $25$ & $197040$\\
$10$ & $500$ & $26$ & $198480$\\
$11$ & $582$ & $27$ & $199912$\\
$12$ & $840$ & $28$ & $204188$\\
$13$ & $1154$ & $29$ & $207930$\\
$14$ & $1606$ & $30$ & $219008$\\
$15$ & $2564$ & $31$ & $230872$\\
$16$ & $4320$ & $32$ & $276032$
\end{tabular}
\vspace{0.0cm} 
\caption{\leftskip 25pt\rightskip 25pt 
The best lower bounds known for kissing numbers in up to
thirty-two dimensions. Numbers in bold are known to be optimal
(see Sch\"utte and van\break{} 
der Waerden~{\protect\cite{SW}},
Leven{\v{s}}te{\u\i}n~{\protect\cite{L}}, Odlyzko and
Sloane~{\protect\cite{OS}} and Musin~{\protect\cite{Musin}}).}
\label{table:kissing}
\end{table}

Once we have developed the basic theory of rigidity for spherical
codes, we will devote the rest of this paper to applying it to analyze
specific codes. We will focus primarily on kissing configurations
(that is, spherical codes with minimal angle at least $\pi/3$, or
equivalently the points of tangency in Euclidean space packings),
because they form a rich class of spherical codes and include many of
the most noteworthy examples.  An \emph{optimal kissing configuration}
is one with the largest possible size in its dimension.

As mentioned above, the face-centered cubic kissing configuration is
not rigid, but we will prove that all of the other best configurations
known in up to twelve dimensions are rigid.  Along the way, we
will produce what may be the first exhaustive enumeration of these
configurations in up to eight dimensions, as well as a complete list of the
known examples in nine through twelve dimensions (although we suspect
that more remain to be discovered).  Above twelve dimensions,
the calculations become increasingly
difficult to do, even by computer, but we analyze certain cases that are
susceptible to conceptual arguments.  In particular, we show that the
kissing configuration of the\break 
Coxeter--Todd lattice $K_{12}$ is not rigid, while that of the
Barnes--Wall lattice $\Lambda_{16}$ is, although both lattices
are conjectured to be optimal sphere packings in their
dimensions.

We particularly focus our attention on $25$ through $31$ dimensions,
because of the\break 
remarkably small increases in the record kissing
numbers from each dimension to the\break 
next (see \fullref{table:previouskiss} in
\fullref{section:kissing25to31}
for the old records). The best configurations previously\pagebreak{} 
known were not even locally jammed, but we see no simple way to
deform them so as to increase the kissing number.  However, in
\fullref{section:kissing25to31} we show how to improve on the
known records.  We give a simple argument that shows how to
beat them, as well as a\break{} 
more complicated construction that makes use of a computer
search to optimize the resulting bounds.

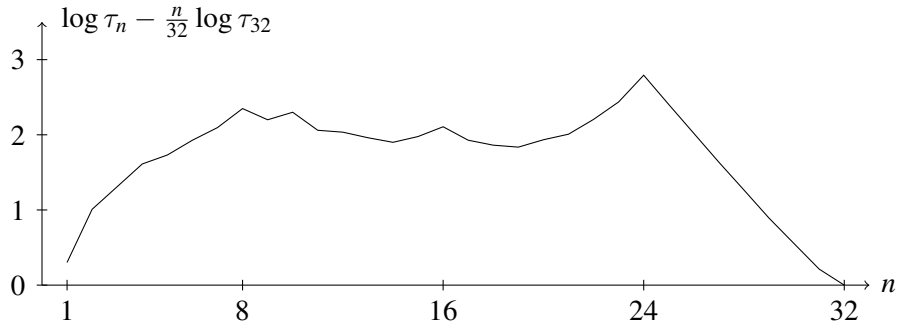
\begin{figure}[ht!]
\begin{center}
\begin{tikzpicture}[xscale=0.3333]
\draw plot coordinates {(1, 0.3016) (2, 1.009) (3, 1.310) (4, 1.612)
  (5, 1.731) (6, 1.928) (7, 2.096) (8, 2.349) (9, 2.200) (10, 2.300)
  (11, 2.060) (12, 2.035) (13, 1.961) (14, 1.900) (15, 1.977) (16,
  2.107) (17, 1.928) (18, 1.862) (19, 1.836) (20, 1.934) (21, 2.008)
  (22, 2.205) (23, 2.437) (24, 2.793) (25, 2.403) (26, 2.019) (27,
  1.635) (28, 1.265) (29, 0.8912) (30, 0.5516) (31, 0.2129) (32, 0)};
\draw[->] (0,0) -- (33,0) node[right] {$n$} coordinate(x axis);
\draw[->] (0,0) -- (0,3.5) node[right] {\ $\log \tau_n - \frac{n}{32} \log
  \tau_{32}$} coordinate(y axis); 
\foreach \x in {1,8,16,24,32}
  \draw (\x cm,2pt) -- (\x cm,-2pt) node[anchor=north] {$\x$};
\foreach \x in {0,1,2,3}
  \draw (-6pt,\x cm) -- (6pt,\x cm) node[anchor=east] {$\x\,\,$}; 
\end{tikzpicture}
\end{center}
\caption{\leftskip 35pt\rightskip 35pt 
A plot of $\log \tau_n - \frac{n}{32} \log \tau_{32}$, where
$\tau_n$ denotes the current record kissing number in $\R^n$}
\label{figure:kissingplot}
\end{figure}

The new records are shown in \fullref{table:kissing}. It is taken
from Conway and Sloane~\cite[page~xxi, Table~I.2(a)]{SPLAG}, with three
exceptions: the entry for $\R^{15}$ was out of date in that table (see
\cite[Chapter~5, Section~4.3]{SPLAG}), the entries for $\R^{13}$
and $\R^{14}$
come from Zinoviev and Ericson~\cite{ZE}, and the entries for $\R^{25}$ through $\R^{31}$
are new results in the present paper.  See also Nebe and Sloane~\cite{NS}.  For the
best upper bounds known in up to twenty-four dimensions, see Mittelmann
and Vallentin~\cite{MV}.

\fullref{figure:kissingplot} shows a logarithmic plot of the
data from \fullref{table:kissing}, normalized for comparison
with $32$ dimensions.  One can see the local maxima
corresponding to the remarkable $E_8$, Barnes--Wall and Leech
lattices in dimensions $8$, $16$ and $24$, respectively.  Note
also that the growth rate of the known kissing numbers drops
dramatically after $24$ dimensions.\pagebreak{} 

\section{Infinitesimal jamming} \label{section:infinijam}

We know of no efficient way to test whether a given spherical code is
jammed.  In principle, it can be done by a finite calculation, at least
if the points in the code have\break 
algebraic numbers as coordinates, by using quantifier
elimination for the first-order theory of the real numbers (see
Tarski~\cite{T}).  (The proof relies on Roth and
Whiteley~\cite[Proposition~3.2]{RW}.) However, quantifier
elimination is not practical in this case.

On the other hand, there are much more efficient tests for a
related concept called infinitesimal jamming, using linear
programming (see Donev, Torquato, Stillinger and
Connelly~\cite{DTSC}). Given a code $\{x_1,\dots,x_N\} \subset
S^{n-1}$, imagine perturbing $x_i$ to $x_i + \varepsilon y_i$.
Then
$$
|x_i + \varepsilon y_i|^2 = 1 + 2  \langle x_i,y_i \rangle \varepsilon + O(\varepsilon^2),
$$
where $\langle\cdot,\cdot\rangle$ denotes the inner product, and
$$
\langle x_i + \varepsilon y_i, x_j + \varepsilon y_j \rangle =
\langle x_i,x_j \rangle +
(\langle x_i,y_j \rangle + \langle x_j,y_i \rangle)\varepsilon  + O(\varepsilon^2).
$$
Thus, to preserve all the constraints up to first order in
$\varepsilon$, we must have $\langle x_i, y_i \rangle = 0$ for\break 
all $i$, and $\langle x_i, y_j \rangle + \langle x_j, y_i
\rangle \le 0$ whenever $\langle x_i, x_j \rangle$ equals the
maximal inner product in the code.  An \emph{infinitesimal
deformation} of the code $\{x_1,\dots,x_N\}$ is a collection\break 
of vectors $y_1,\dots,y_N$ satisfying these constraints.  It is
an \emph{infinitesimal rotation} if\break 
there exists a skew-symmetric matrix $\Phi \in \R^{n \times n}$
such that $y_i = \Phi x_i$ for all $i$, and a\break 
code is \emph{infinitesimally jammed} if every infinitesimal
deformation is an infinitesimal rotation. (Recall that the
skew-symmetric matrices are exactly those in the Lie algebra of
$SO(n)$.)  Note that for an infinitesimal rotation,
$$
\langle x_i,y_j \rangle = \langle x_i,\Phi x_j \rangle
= - \langle \Phi x_i,x_j \rangle = - \langle x_j, y_i\rangle.
$$
Thus, an infinitesimal rotation does not change any distances,
up to first order.  The converse is false (consider a square on
the equator in $S^2$, with an infinitesimal de\-formation 
moving two opposite corners up and the other two down), but it
is true for full-dimensional codes:

\begin{lemma} \label{lemma:spanning}
Let $y_1,\dots,y_N$ be an infinitesimal deformation of a code
$\{x_1,\dots,x_N\}$\break 
in $S^{n-1}$ such that $x_1,\dots,x_N$ span $\R^n$. If $\langle
x_i,y_j \rangle + \langle x_j,y_i \rangle = 0$ for all $i$ and
$j$,\break 
then the deformation is an infinitesimal rotation.
\end{lemma}

\begin{proof}
First, note that if a linear combination $\sum_i \alpha_i x_i$
vanishes, then $\sum_i \alpha_i y_i = 0$ as well, because
\begin{align*}
\bigg\langle \sum_i \alpha_i y_i, x_j \bigg\rangle & =
 \sum_i \alpha_i \langle y_i, x_j \rangle\\
 &  =
- \sum_i \alpha_i \langle x_i, y_j \rangle\\
&  =
- \bigg\langle \sum_i \alpha_i x_i, y_j \bigg\rangle\\
&  = - \langle \hspace{1pt} 
0, y_j \rangle = 0
\end{align*}
for all $j$ (and only the zero vector is orthogonal to a set that spans
$\R^n$).  Thus, there is a well-defined linear map $\Phi$ such that
$\Phi x_i = y_i$.  Furthermore, the identity
$$
\langle x_i, \Phi x_j \rangle = \langle x_i, y_j\rangle = - \langle y_i, x_j \rangle
= - \langle \Phi x_i, x_j \rangle
$$
implies that $\Phi$ is skew-symmetric, because it holds for a basis of
$\R^n$ and hence $\langle u, \Phi v \rangle = - \langle \Phi u, v
\rangle$ for all $u,v \in \R^n$.
\end{proof}

Every infinitesimally jammed code is in fact jammed.  This is not
obvious: one cannot simply differentiate a purported unjamming motion
to get an infinitesimal unjamming, without dealing with two
technicalities, namely whether there is a differentiable unjam\-
ming and what happens if all the first-order derivatives
vanish. However, it is true, as pointed out by
Connelly~\cite[Remark~4.1]{Con} and by Roth and
Whiteley~\cite[Theorem~5.7]{RW}:

\begin{theorem}[Connelly, Roth and Whiteley]
Every infinitesimally jammed spherical code is jammed.
\end{theorem}

The cited papers deal with the more general setting of tensegrity
frameworks, in which movable points can be connected by \emph{bars}
(with fixed lengths), \emph{cables} (with specified maximum lengths) or
\emph{struts} (with specified minimum lengths), and they prove
that\pagebreak{}\break{}
infinitesimal jamming implies jamming in this setting.  For the special
case of spherical codes, we connect each point in the code to the origin
using a bar, and we insert struts between neighboring points (that is, those
at the minimal distance).

We do not know whether every jammed spherical code that spans the ambient
space\break
is infinitesimally jammed.  For tensegrity frameworks, the
corresponding statement\break
is not true: if we place bars along the edges of a regular
octahedron, and use two additional bars to connect its center
with a pair of opposite vertices, then the framework is rigid,
but flexing the center orthogonally to the two adjacent bars is
a nontrivial infinitesimal deformation. We have not found such
an example for spherical codes, but we expect that there is
one.  By contrast, infinitesimal jamming is equivalent to
jamming for periodic packings in Euclidean space (see Donev,
Torquato, Stillinger and Connelly~\cite{DTSC}).  Specifically,
if we perturb $x$ to $x+\eps y$ and $x'$ to $x'+\eps y'$, then
$$
|(x + \varepsilon y) - (x' + \varepsilon y')|^2 = |x-x'|^2 +
2 \langle x-x', y-y' \rangle \varepsilon + |y-y'|^2 \varepsilon^2.
$$
The second-order term is always nonnegative, so nonnegativity of the
first-order term\break
suffices to produce an actual unjamming.
(De\nobreak{}form\nobreak{}ing\nobreak{}
the\nobreak{} un\nobreak{}der\nobreak{}ly\nobreak{}ing\nobreak{} lat\nobreak{}tice\nobreak{} com\nobreak{}pli\nobreak{}cates\break
the analysis, but the result remains true;
see~\cite[Appendix~C]{DTSC}.) What goes wrong in the spherical
case is that $x+\varepsilon y$ is no longer a unit vector and
must be normalized, which causes the distances to decrease.

This is not merely a technicality: there seems to be no simple method
to turn an infinitesimal unjamming into an actual unjamming.
Nevertheless, in all our examples, we have been able to accomplish this
(with some effort).

The linear programming algorithm for infinitesimal rigidity testing
works as follows. By \fullref{lemma:spanning}, to test whether a
full-dimensional code is infinitesimally jammed, we need only check
for each pair of points whether the distance between them can be
changed.  In other words, in an infinitesimal deformation, are the
maximum and minimum of $\langle x_i,y_j \rangle + \langle x_j,y_i
\rangle$ zero for all $i$ and $j$?  For each $i$ and $j$, this gives
rise to two linear programming problems, because we are imposing
linear constraints on the perturbation vectors $y_1,\dots,y_N$ and
maximizing or minimizing the linear function $\langle x_i,y_j \rangle
+ \langle x_j,y_i \rangle$.  (Of course, when $\langle x_i, x_j
\rangle$ is maximal in the code, the definition\break
of an infinitesimal deformation requires that $\langle x_i,y_j
\rangle + \langle x_j, y_i \rangle \le 0$, so maximizing this
linear functional is trivial. However, the other cases are
nontrivial.)  The code is\break
infinitesimally jammed if and only if the optima in all these
linear programs are zero.  If not, then solving the linear
programs will produce an infinitesimal unjamming, provided that
we also bound the coordinates of the perturbation vectors (to
avoid unbounded linear programs).

\enlargethispage{0.12cm} 

\nopagebreak 
In some cases we are aided by symmetry, because we only need to
check one \nopagebreak{}representa\-\nopagebreak{}tive\nopagebreak{} 
from each orbit of the action of the code's symmetry group on
pairs of points in the\pagebreak{}\break{}
code. For example, if the symmetry group acts
distance-transitively, then we only need to check one pair of
points at each distance. Using the approach of Donev, Torquato,
Stillinger and Connelly~\cite{DTSC}, we can even reduce to
solving one linear program, at the cost of randomization.
Specifically, consider maximizing the linear combination
$$
\sum_{i,j} c_{i,j} (\langle x_i,y_j \rangle + \langle x_j,y_i
\rangle),
$$
where the coefficients $c_{i,j}$ are chosen randomly from the interval
$[-1,1]$. With probability $1$, this approach will produce an
infinitesimal unjamming if one exists.  Thus, if the optimum is zero,
then we can be confident that the code is jammed, although this does
not constitute a proof.

For most of the examples in this paper, we give short conceptual proofs
of jamming.  However, for some cases we must rely on computer
calculations.  In these cases, we
have given rigorous, computer-assisted proofs by using exact rational
arithmetic via the \texttt{QSopt\_ex} linear programming software
of Applegate, Cook, Dash and Espinoza~\cite{ACDE} and checking every pair of points in the code.

\section{The kissing configurations of root lattices}
\label{rootlattices}

We begin by proving that the root systems $D_n$ (for $n \ge 4$) and
$E_6$, $E_7$ and $E_8$ are infinitesimally jammed, while $A_n$ is not
jammed (for $n \ge 3$). Note that these root systems are the kissing
configurations of the corresponding root lattices.

In this section, we will use spheres of diameter $\sqrt{2}$ instead of
$1$, because that is\break
standard for these root systems and makes the inner products
integral.  Note that the theory of infinitesimal jamming in no
way depends on this normalization.

The following elementary lemma will play a key role in the proofs:

\begin{lemma} \label{lemma:reducedim}
Let $\mathcal{C}$ and $\mathcal{D}$ be spherical codes with
$\mathcal{C} \subseteq \mathcal{D}$ and with the same minimal distance.
If $\mathcal{C}$ is infinitesimally jammed within the vector space it
spans, then in any infinitesimal deformation of $\mathcal{D}$, the
inner products between points in $\mathcal{C}$ are unchanged (up to
first order).
\end{lemma}

The interesting case is when $\mathcal{C}$ is lower dimensional than
$\mathcal{D}$.

\begin{proof}
Let $x$ and $y$ be points in $\mathcal{C}$, and let $u$ and $v$ be
their perturbations in an infinitesimal deformation of $\mathcal{D}$.
We write $u=u_{\mathcal{C}} + u_\perp$ and $v=v_{\mathcal{C}} + v_\perp$,
where $u_{\mathcal{C}}$ and $v_{\mathcal{C}}$ are in the span of
$\mathcal{C}$ while $u_\perp$ and $v_\perp$ are in the orthogonal
complement of the span.\pagebreak

The orthogonal projections to the span of $\mathcal{C}$ yield an
infinitesimal deformation of $\mathcal{C}$.  (Here we need
$\mathcal{C}$ and $\mathcal{D}$ to have the same minimal distance,
since otherwise the\break
conditions on which inner products can increase will differ.)
Thus, because $\mathcal{C}$ is infinitesimally jammed within
its span, $\langle x, v_{\mathcal{C}} \rangle + \langle
u_{\mathcal{C}}, y \rangle = 0$. Furthermore, $\langle x,
v_\perp \rangle = \langle y, u_\perp \rangle = 0$.  It follows
that $\langle x, v \rangle + \langle u, y \rangle = 0$, as
desired.
\end{proof}

\begin{lemma}
The $A_2$ root system is infinitesimally jammed.
\end{lemma}

The $A_2$ root system is a regular hexagon, and it is easy to show that
every regular polygon is infinitesimally jammed.  This simple
observation provides a useful tool for analyzing more elaborate
configurations via \fullref{lemma:reducedim}.

\begin{proposition} \label{prop:D4jammed}
The $D_4$ root system is infinitesimally jammed.
\end{proposition}

\begin{proof}
The minimal vectors of $D_4$ have norm $2$ and the possible inner
products between distinct minimal vectors are $0$, $\pm 1$ and $-2$.
First, note that the automorphism group of $D_4$ acts transitively on
pairs of minimal vectors with a given inner product, so without loss of
generality we can consider just one pair of points at each distance.
(This transitivity fails for $D_n$ with $n>4$, because there are two
orbits for inner product $0$, but the triality symmetry of $D_4$
collapses them to one orbit.)

Furthermore, $D_4$ contains $A_2$, and by \fullref{lemma:reducedim}
the distances in a copy of $A_2$ cannot change because $A_2$ is
infinitesimally jammed within its span.  This takes care of all the
cases except for a pair of orthogonal vectors.

We now have to show that if $\langle x,y \rangle = 0$ then $\langle x,y
\rangle$ does not change in any infinitesimal deformation.  Again by
the distance transitivity of the automorphism group, we may assume $x =
(1,1,0,0)$ and $y = (1,-1,0,0)$. Let $u = (1,0,1,0)$, $v = (1,0,-1,0)$,
$w = (1,0,0,1)$ and $z = (1,0,0,-1)$ be other minimal vectors of
$D_4$. Denote the\break
first order perturbations of $x,y,u,v,w,z$ by
$x',y',u',v',w',z'$. Note that $x + y = u + v = w + z$, and
that $\langle x,u \rangle = \langle x,v \rangle = \langle y,u
\rangle = \langle y,v \rangle = 1$. Thus, by the $A_2$
embedding argument, we have
\begin{align*}
\langle x, u' \rangle + \langle x', u \rangle &= 0, &
\langle y, u' \rangle + \langle y', u \rangle &= 0, \\
\langle x, v' \rangle + \langle x', v \rangle &= 0, &
\langle y, v' \rangle + \langle y', v \rangle &= 0.
\end{align*}
Adding these equations, we get
$$
\langle x + y, u' + v' \rangle + \langle x' + y', u + v \rangle = 0,
$$
or (using $x + y = u + v$)
$$
\langle u + v, u' + v' \rangle + \langle x' + y', x + y \rangle = 0.
$$
Since we know that $\langle u,u' \rangle = 0$, etc., we get (denoting
the first order change $\langle u, v' \rangle + \langle u', v \rangle$
in $\langle u,v \rangle$ by $\delta(u,v)$)
$$
\delta(u,v)  + \delta(x,y)= 0.
$$
Similarly, we have
\begin{align*}
\delta(u,v)  + \delta(w,z) &= 0, \\
\delta(w,z)  + \delta(x,y) &= 0.
\end{align*}
From these three equations, elementary algebra implies $\delta(x,y) =
\delta(u,v) = \delta(z,w) = 0$. This completes the proof.
\end{proof}

The lengthy argument for the last case amounts to verifying that a
square embedded within $D_4$ cannot be infinitesimally deformed.  Note
that this cannot simply be settled using \fullref{lemma:reducedim},
although the square is indeed jammed within its span, because the
minimal distance in the square differs from that in $D_4$. If that
argument worked, it would also prove infinitesimal jamming for $D_3$,
which is not true.  (The $A_3$ and $D_3$ root lattices are isomorphic
to the face-centered cubic lattice, whose kissing configuration is not
jammed.)

\begin{corollary} \label{cor:rootjammed}
The $D_n$ root system (for $n \ge 4$) and the $E_6$, $E_7$ and $E_8$
root systems are infinitesimally jammed.
\end{corollary}

\begin{proof}
These configurations have norm $2$ and inner products $0$, $\pm1$ and
$\pm 2$, the same as in $D_4$.  We first deal with $E_6$, $E_7$ and
$E_8$.  Their automorphism groups act distance transitively, so it
suffices to consider a single pair of points at each distance. The
$D_4$ root system embeds in each of these configurations (in fact, its
Dynkin diagram is a subdiagram), so without loss of generality we can
assume the pair of points is in $D_4$.  Now combining
\fullref{lemma:reducedim} and \fullref{prop:D4jammed}
completes the proof.

The same proof works for $D_n$ with $n > 4$, with one exception, namely
that there are two orbits of pairs of orthogonal vectors, so the group
does not quite act distance transitively. Specifically, the stabilizer
of $(1,1,0,\dots,0)$ cannot interchange $(1,-1,0,\dots,0)$ and
$(0,0,1,1,0,\dots,0)$. However, in both cases, these vectors are
contained in a copy of $D_4$ (namely, the one in the first four
coordinates), so we can complete the proof as before.
\end{proof}

The $A_n$ root system is locally jammed, and for $n=2$ it is in fact
jammed, but the unjamming for $n=3$ extends to higher dimensions.

\begin{proposition} \label{prop:Anunjammed}
For $n \ge 3$, the $A_n$ root system is not jammed.
\end{proposition}

\pagebreak

\begin{proof}
We will demonstrate an explicit unjamming.  We realize the $A_n$
lattice as the set of integral vectors in the subspace $\{(x_0, \dots,
x_n) : \sum_i x_i = 0 \}$ of $\R^{n+1}$, so each vector in the $A_n$
root system has one $1$ and one $-1$ among its coordinates.

We begin with $A_3$, whose twelve vertices form a cuboctahedron with
six square facets and eight triangular facets, and we choose a perfect
matching using non-overlapping diagonals of the squares.  In the left
half of the diagram below, we have labeled matched vertices with the
same label (from $1$ to $6$).  Each matched pair $v$ and $v'$ satisfies
$\langle v,v' \rangle = 0$, so we can perturb $v$ to $(v+\eps
v')/\sqrt{1+\eps^2}$ and $v'$ to $(v'+\eps v)/\sqrt{1+\eps^2}$,\break
as shown on the right.

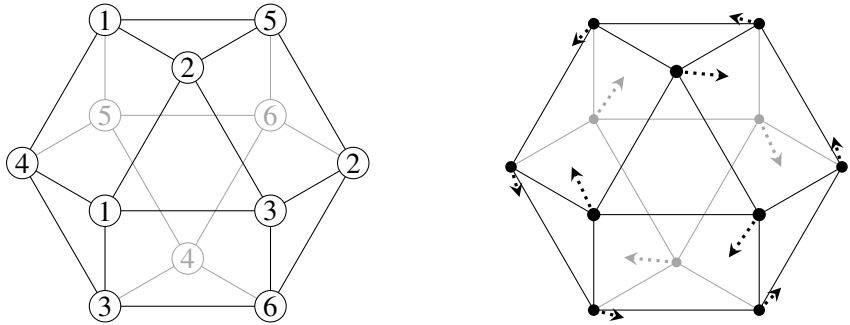
\begin{center}
\begin{tikzpicture}[scale=2.2,>=stealth]
\fill[black!35] (1, 0.57735027) circle (0.03);
\fill[black!35] (0, 0.57735027) circle (0.03);
\fill[black!35] (0.5, -0.28867513) circle (0.03);
\draw[black!35] (1, 0.57735027) -- (0, 0.57735027) -- (0.5, -0.28867513) -- (1, 0.57735027);
\draw[black!35] (1, 0.57735027) -- (1.5, 0.28867513);
\draw[black!35] (1, 0.57735027) -- (1, 1.1547005);
\draw[black!35] (0, 0.57735027) -- (0, 1.1547005);
\draw[black!35] (0, 0.57735027) -- (-0.5, 0.28867513);
\draw[black!35] (0.5, -0.28867513) -- (0, -0.57735027);
\draw[black!35] (0.5, -0.28867513) -- (1, -0.57735027);
\draw (1, 0) -- (1.5, 0.28867513);
\draw (0.5, 0.86602540) -- (1, 1.1547005);
\draw (0.5, 0.86602540) -- (0, 1.1547005);
\draw (0, 0) -- (-0.5, 0.28867513);
\draw (0, 0) -- (0, -0.57735027);
\draw (1, 0) -- (1, -0.57735027);
\draw (1.5, 0.28867513) -- (1, 1.1547005) -- (0, 1.1547005) -- (-0.5, 0.28867513) -- (0, -0.57735027) -- (1, -0.57735027) -- (1.5, 0.28867513);
\draw (0, 0) -- (1, 0) -- (0.5, 0.86602540) -- (0,0);
\fill[black!0] (0, 0) circle (0.095);
\fill[black!0] (1, 0) circle (0.095);
\fill[black!0] (0.5, 0.86602540) circle (0.095);
\draw (0, 0) circle (0.095);
\draw (1, 0) circle (0.095);
\draw (0.5, 0.86602540) circle (0.095);
\draw (0, 0) node {$1$};
\draw (1, 0) node {$3$};
\draw (0.5, 0.86602540) node {$2$};
\fill[black!0] (1, 0.57735027) circle (0.095);
\fill[black!0] (0, 0.57735027) circle (0.095);
\fill[black!0] (0.5, -0.28867513) circle (0.095);
\draw[black!35] (1, 0.57735027) circle (0.095);
\draw[black!35] (0, 0.57735027) circle (0.095);
\draw[black!35] (0.5, -0.28867513) circle (0.095);
\draw[black!35] (1, 0.57735027) node {$6$};
\draw[black!35] (0, 0.57735027) node {$5$};
\draw[black!35] (0.5, -0.28867513) node {$4$};
\fill[black!0] (1.5, 0.28867513) circle (0.095);
\fill[black!0] (1, 1.1547005) circle (0.095);
\fill[black!0] (0, 1.1547005) circle (0.095);
\fill[black!0] (-0.5, 0.28867513) circle (0.095);
\fill[black!0] (0, -0.57735027) circle (0.095);
\fill[black!0] (1, -0.57735027) circle (0.095);
\draw (1.5, 0.28867513) circle (0.095);
\draw (1, 1.1547005) circle (0.095);
\draw (0, 1.1547005) circle (0.095);
\draw (-0.5, 0.28867513) circle (0.095);
\draw (0, -0.57735027) circle (0.095);
\draw (1, -0.57735027) circle (0.095);
\draw (1.5, 0.28867513) node {$2$};
\draw (1, 1.1547005) node {$5$};
\draw (0, 1.1547005) node {$1$};
\draw (-0.5, 0.28867513) node {$4$};
\draw (0, -0.57735027) node {$3$};
\draw (1, -0.57735027) node {$6$};
\end{tikzpicture}
\qquad \qquad
\begin{tikzpicture}[scale=2.2,>=stealth]
\fill[black!0] (1.5, 0.28867513) circle (0.095);
\fill[black!0] (1, 1.1547005) circle (0.095);
\fill[black!0] (0, 1.1547005) circle (0.095);
\fill[black!0] (-0.5, 0.28867513) circle (0.095);
\fill[black!0] (0, -0.57735027) circle (0.095);
\fill[black!0] (1, -0.57735027) circle (0.095);
\fill[black!35] (1, 0.57735027) circle (0.03);
\fill[black!35] (0, 0.57735027) circle (0.03);
\fill[black!35] (0.5, -0.28867513) circle (0.03);
\draw[black!35] (1, 0.57735027) -- (0, 0.57735027) -- (0.5, -0.28867513) -- (1, 0.57735027);
\draw[black!35] (1, 0.57735027) -- (1.5, 0.28867513);
\draw[black!35] (1, 0.57735027) -- (1, 1.1547005);
\draw[black!35] (0, 0.57735027) -- (0, 1.1547005);
\draw[black!35] (0, 0.57735027) -- (-0.5, 0.28867513);
\draw[black!35] (0.5, -0.28867513) -- (0, -0.57735027);
\draw[black!35] (0.5, -0.28867513) -- (1, -0.57735027);
\draw (1, 0) -- (1.5, 0.28867513);
\draw (0.5, 0.86602540) -- (1, 1.1547005);
\draw (0.5, 0.86602540) -- (0, 1.1547005);
\draw (0, 0) -- (-0.5, 0.28867513);
\draw (0, 0) -- (0, -0.57735027);
\draw (1, 0) -- (1, -0.57735027);
\fill (1.5, 0.28867513) circle (0.035);
\fill (1, 1.1547005) circle (0.035);
\fill (0, 1.1547005) circle (0.035);
\fill (-0.5, 0.28867513) circle (0.035);
\fill (0, -0.57735027) circle (0.035);
\fill (1, -0.57735027) circle (0.035);
\draw (1.5, 0.28867513) -- (1, 1.1547005) -- (0, 1.1547005) -- (-0.5, 0.28867513) -- (0, -0.57735027) -- (1, -0.57735027) -- (1.5, 0.28867513);
\fill (0, 0) circle (0.04);
\fill (1, 0) circle (0.04);
\fill (0.5, 0.86602540) circle (0.04);
\draw (0, 0) -- (1, 0) -- (0.5, 0.86602540) -- (0,0);
\draw[->, very thick, dotted] (0, 0) -- (-0.13245553, 0.28867513);
\draw[->, very thick, dotted] (1.0000000, 0) -- (0.81622777, -0.25904742);
\draw[->, very thick, dotted] (0.50000000, 0.86602540) -- (0.81622777, 0.83639769);
\draw[->, very thick, dotted] (1.5000000, 0.28867513) -- (1.4486833, 0.47124932);
\draw[->, very thick, dotted] (0, -0.57735027) -- (0.18377223, -0.62419579);
\draw[->, very thick, dotted] (0, 1.1547005) -- (-0.13245553, 1.0189719);
\draw[->, very thick,black!35, dotted] (1.0000000, 0.57735027) -- (1.1324555, 0.28867513);
\draw[->, very thick,black!35, dotted] (0, 0.57735027) -- (0.18377223, 0.83639769);
\draw[->, very thick,black!35, dotted] (0.50000000, -0.28867513) -- (0.18377223, -0.25904742);
\draw[->, very thick, dotted] (-0.50000000, 0.28867513) -- (-0.44868330, 0.10610095);
\draw[->, very thick, dotted] (1.0000000, 1.1547005) -- (0.81622777, 1.2015461);
\draw[->, very thick, dotted] (1.0000000, -0.57735027) -- (1.1324555, -0.44162161);
\end{tikzpicture}
\end{center}

To see that this is an actual unjamming, we must check what happens to
vertices $v$ and $w$ satisfying $\langle v,w \rangle = 1$.  The matched
vertices satisfy $\langle v',w' \rangle = -1$ and $\langle v',w \rangle
= - \langle v,w' \rangle$, and hence
$$
\frac{\langle v + \eps v', w + \eps w' \rangle}{1+\eps^2} =
\frac{1-\eps^2}{1+\eps^2} < 1,
$$
as desired.

For $A_n$ with $n > 3$, we simply perform this unjamming on the copy of
$A_3$ in the first four coordinates, while leaving all the other points
unchanged. All we need to check is\break
that the minimal distances between the moving points and the
unchanged points do not\break
decrease.  Suppose $v$ is in $A_3$, paired with $v'$ as above,
while $w$ is outside $A_3$ and satisfies $\langle v,w \rangle =
1$. Then $v$ and $w$ have a single nonzero coordinate in which
they agree, and the other nonzero coordinate of $w$ is not
among the first four coordinates.  Because $\langle v,v'
\rangle = 0$, the vectors $v$ and $v'$ have disjoint supports,
from which it follows that $\langle v',w \rangle=0$. Thus,
$$
\frac{\langle v + \eps v',w \rangle}{\sqrt{1+\eps^2}} =
\frac{1}{\sqrt{1+\eps^2}} < 1,
$$
so we have a genuine unjamming of the $A_n$ root system.
\end{proof}

Note that the first-order changes in the inner products of $1$ vanish,
as they must according to \fullref{lemma:reducedim} (because each
pair of vectors with inner product $1$ spans a hexagon); instead, there
is a second-order improvement.

\vspace{-0.3cm} 
\section{Kissing configurations in five, six and seven dimensions}
\label{section:kisslist} \vspace{-0.2cm} 

\enlargethispage{0.35cm} 

The root systems analyzed in the previous section are
conjectured to be optimal kissing configurations up through
eight dimensions.  In fact, $E_8$ is known to be the unique
optimal kissing configuration in $\R^8$ (optimality was proved
by Leven{\v{s}}te{\u\i}n~\cite{L} and\break
by Odlyzko and Sloane~\cite{OS}, while uniqueness was proved by
Bannai and Sloane~\cite{BS}).  However, in five, six and seven
dimensions, the $D_5$, $E_6$ and $E_7$ root systems are not the
unique optimal kissing configurations. This was first observed
by Leech~\cite{Lee1,Lee2}.

In this section, we enumerate what we believe may be a complete list of
optimal kissing configurations in these dimensions, up to isometry.  We
find two configurations in $\R^5$, four in $\R^6$, and four in $\R^7$.
Those in $\R^5$ and $\R^6$, as well as two of the ones in $\R^7$,\break
are all kissing configurations of previously known sphere
packings. However, two of the configurations in $\R^7$ do not
arise naturally from densest packings in $\R^7$ and appear to
be new.

We have verified by rigorous computation calculations that all these
configurations are infinitesimally jammed.  (It would very likely be
possible to check this by hand, but it\break
does not seem worth the effort.) To do so, we used version~2.6
of the \texttt{QSopt\_ex} rational\break
LP solver~\cite{ACDE} to get exact solutions to the linear
programs. Specifically, for each pair of points, we used the
software to verify that their inner product can neither
increase nor decrease in any infinitesimal deformation.  Of
course, it would be far more efficient to take into account the
orbits of the automorphism group on pairs of points, but it was
easier to let the computer check many cases than to compute and
keep track of the orbits. The one tricky part of the
calculation is that the coordinates of the points are not
always rational numbers (and using floating-point arithmetic
would make the calculations unrigorous). Fortunately, there is
an easy fix: by rescaling certain coordinates, in each case we
can use rational coordinates and carry out all the calculations
with respect to a nonstandard inner product defined by a
rational matrix. In fact, we can take the matrix to be
diagonal.  In the supplementary information for this paper (see
\fullref{appendix:data}), we provide explicit rational
coordinates for all of the configurations studied in this
section, together with the corresponding inner product
matrices.\nopagebreak

\vspace{-0.125cm} 
\subsection{Methods}
\vspace{-0.075cm} 

Conway and Sloane~\cite{Allthebest} describe a method that
conjecturally generates all the best sphere packings in low
dimensions.
For dimensions from five to eight, it works as\pagebreak{}\break 
follows. We start with the checkerboard lattice
$$
D_4 = \bigg\{x \in \Z^4 : \sum_{i=1}^4 x_i \equiv 0 \pmod 2\bigg\}.
$$
It leads to a sphere packing in $\R^4$ with balls of radius
$\sqrt{2}/2$, which is conjecturally the densest sphere packing in four
dimensions, as well as the unique densest periodic packing. The
discriminant group $D_4^*/D_4 \cong (\Z/2\Z)^2$ has four elements, the
zero class $a = (0,0,0,0)$ and three nonzero classes, which we
represent by the deep holes $b = (1,0,0,0)$, $c =
\left(1/2,1/2,1/2,1/2\right)$ and $d = \left(-1/2,1/2,1/2,1/2\right)$.
Conway and Sloane conjecture that the densest packings in dimension $n$
(with $5 \le n \le 8$) are\break
obtained by \emph{fibering} over $D_4$, that is, by positioning
translated copies of $D_4$ over the densest packings in
dimension $n-4$, scaled so that their minimal distance is $1$.
Each translation vector will be one of $a$, $b$, $c$ or $d$,
and adjacent points in the $(n{-}4)$--dimensional packing will
be assigned different translation vectors. This idea of
fibering is a generalization of the construction of laminated
lattices (see Conway and Sloane~\cite{Laminated}).  See also
Cohn and Kumar~\cite{Groundstates} for other calculations using
the methods of \cite{Allthebest}.

\subsection{Dimension $5$: $40$ points on $S^4$}

In dimension $5$, the densest packings are conjecturally obtained by
stacking layers\break
of $D_4$ on top of each other, that is, by arranging translated
copies of $D_4$ along the integers $\Z$.  At each point of
$\Z$, one has to make a choice of which translate of $D_4$ to
use. In other words, such a packing corresponds to a coloring
of $\Z$ by four colors or labels $a,b,c,d$, so that no two
adjacent integers have the same color.

For instance, the $D_5$ lattice is obtained from the following
coloring: 
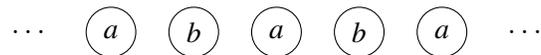
\setlength{\unitlength}{1cm}
\begin{center}
\begin{tikzpicture}[scale=1.1]
\draw (0,0) node {$\ldots$};
\draw (6,0) node {$\ldots$};
\draw (1,0) circle (0.3);
\draw (2,0) circle (0.3);
\draw (3,0) circle (0.3);
\draw (4,0) circle (0.3);
\draw (5,0) circle (0.3);
\draw (1,0) node {{$a$}};
\draw (2,0) node {{$b$}};
\draw (3,0) node {{$a$}};
\draw (4,0) node {{$b$}};
\draw (5,0) node {{$a$}};
\end{tikzpicture}
\end{center}
Because the affine symmetry group of $D_4^*$ acts as the full symmetric
group on the four classes in $D_4^*/D_4$, it does not matter which two
distinct classes are used as labels here.

The best kissing configurations known in dimension $5$ are obtained
from the local versions of these packings.  The $D_5$ root system comes
from the following diagram: 
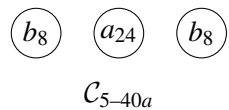
\begin{center}
\begin{tikzpicture}[scale=1.1]
\draw (1,0) circle (0.3);
\draw (2,0) circle (0.3);
\draw (3,0) circle (0.3);
\draw (1,0) node {$b_8$};
\draw (2,0) node {$a_{24}$};
\draw (3,0) node {$b_8$};
\draw (2,-0.75) node {$\sCl{5}{40a}$};
\end{tikzpicture}
\vspace{-0.5cm}
\end{center}
The subscripts mean that we take the $24$ minimal lattice vectors in
the central $D_4$ layer, and $8$ translated lattice vectors in each
adjacent layer (specifically, those surrounding a\break{} 
deep hole). We call this configuration $\sCl{5}{40a}$.

Of course, there is no need to choose the same translation vector on
both sides of the central layer, and we can form a competing
configuration $\sCl{5}{40b}$ as follows: 
\begin{center}
\begin{tikzpicture}[scale=1.1]
\draw (1,0) circle (0.3);
\draw (2,0) circle (0.3);
\draw (3,0) circle (0.3);
\draw (1,0) node {$b_8$};
\draw (2,0) node {$a_{24}$};
\draw (3,0) node {$c_8$};
\draw (2,-0.75) node {$\sCl{5}{40b}$};
\end{tikzpicture}
\vspace{-0.25cm} 
\end{center}
This is the kissing configuration of the other three uniform
$5$--dimensional packings described in \cite{Allthebest}. It is
genuinely different from $\sCl{5}{40a}$, because it lacks antipodal
symmetry. Its symmetry group has size $384$, compared to $3840$ for
$\sCl{5}{40a}$.

\subsection{Dimension $6$: $72$ points on $S^5$}

Next, we consider kissing configurations in six dimensions.  The best
packings known in\break
$\R^6$ are obtained by fibering over $D_4$ using a hexagonal
arrangement of translates in the plane. There are four uniform
packings, described concisely by the colors associated to the
hexagon of six translates around a central copy of $D_4$. For
more details we refer the reader to Conway and
Sloane~\cite{Allthebest}, and Cohn and
Kumar~\cite{Groundstates}.

These packings have four distinct kissing configurations, shown in
\fullref{fig:hex}, and a simple case analysis shows that every
possibility in this framework is equivalent to one of them (up to
symmetries of the hexagon and permutations of the labels). They are
genuinely different, as their automorphism groups have sizes $103680$,
$3840$, $2304$ and $384$, respectively.

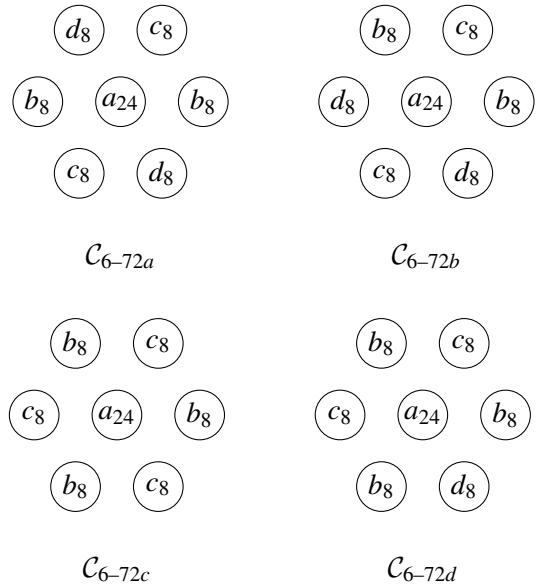
\begin{figure}[ht!]
\centering
\begin{tikzpicture}[scale=1.1]
\draw (1,0) circle (0.3);
\draw (2,0) circle (0.3);
\draw (3,0) circle (0.3);
\draw (1.5,0.866) circle (0.3);
\draw (2.5,0.866) circle (0.3);
\draw (1.5,-0.866) circle (0.3);
\draw (2.5,-0.866) circle (0.3);
\draw (1,0) node {$b_8$};
\draw (2,0) node {$a_{24}$};
\draw (3,0) node {$b_8$};
\draw (1.5,0.866) node {$d_8$};
\draw (2.5,0.866) node {$c_8$};
\draw (1.5,-0.866) node {$c_8$};
\draw (2.5,-0.866) node {$d_8$};
\draw (2,-1.866) node {$\sCl{6}{72a}$} ;
\end{tikzpicture} \hspace{1cm}
\begin{tikzpicture}[scale=1.1]
\draw (1,0) circle (0.3);
\draw (2,0) circle (0.3);
\draw (3,0) circle (0.3);
\draw (1.5,0.866) circle (0.3);
\draw (2.5,0.866) circle (0.3);
\draw (1.5,-0.866) circle (0.3);
\draw (2.5,-0.866) circle (0.3);
\draw (1,0) node {$d_8$};
\draw (2,0) node {$a_{24}$};
\draw (3,0) node {$b_8$};
\draw (1.5,0.866) node {$b_8$};
\draw (2.5,0.866) node {$c_8$};
\draw (1.5,-0.866) node {$c_8$};
\draw (2.5,-0.866) node {$d_8$};
\draw (2,-1.866) node {$\sCl{6}{72b}$} ;
\end{tikzpicture}
\vskip 0.5cm
\begin{tikzpicture}[scale=1.1]
\draw (1,0) circle (0.3);
\draw (2,0) circle (0.3);
\draw (3,0) circle (0.3);
\draw (1.5,0.866) circle (0.3);
\draw (2.5,0.866) circle (0.3);
\draw (1.5,-0.866) circle (0.3);
\draw (2.5,-0.866) circle (0.3);
\draw (1,0) node {$c_8$};
\draw (2,0) node {$a_{24}$};
\draw (3,0) node {$b_8$};
\draw (1.5,0.866) node {$b_8$};
\draw (2.5,0.866) node {$c_8$};
\draw (1.5,-0.866) node {$b_8$};
\draw (2.5,-0.866) node {$c_8$};
\draw (2,-1.866) node {$\sCl{6}{72c}$} ;
\end{tikzpicture} \hspace{1cm}
\begin{tikzpicture}[scale=1.1]
\draw (1,0) circle (0.3);
\draw (2,0) circle (0.3);
\draw (3,0) circle (0.3);
\draw (1.5,0.866) circle (0.3);
\draw (2.5,0.866) circle (0.3);
\draw (1.5,-0.866) circle (0.3);
\draw (2.5,-0.866) circle (0.3);
\draw (1,0) node {$c_8$};
\draw (2,0) node {$a_{24}$};
\draw (3,0) node {$b_8$};
\draw (1.5,0.866) node {$b_8$};
\draw (2.5,0.866) node {$c_8$};
\draw (1.5,-0.866) node {$b_8$};
\draw (2.5,-0.866) node {$d_8$};
\draw (2,-1.866) node {$\sCl{6}{72d}$} ;
\end{tikzpicture}
\vspace{-0.25cm} 
\caption{Kissing configurations in $\R^6$ based on the hexagonal lattice} \label{fig:hex}
\end{figure}

\subsection{Dimension $7$: $126$ points on $S^6$}

The best packings known in $\R^7$ are obtained by fibering over $D_4$
in an optimal $3$--dimensional arrangement, namely one of the Barlow
packings.  The Barlow packings are themselves obtained by stacking
translates of the hexagonal lattice.  Here, we will\break
be concerned with the kissing configurations, which involve
only three hexagonal layers, so although there are uncountably
many Barlow packings (corresponding to three-colorings of
$\Z$), we will only need to consider two configurations: the
hexagonal close-packing, in which the top and bottom layers are
mirror images of each other, and the face-centered cubic
packing, in which they are point reflections of each other.

First, we consider the face-centered cubic arrangement.  Without loss
of generality we label the central sphere at the origin $a$.  Once we
have specified the labels in the\pagebreak{} 
central hexagonal layer, the labels on the minimal vectors in
the adjacent layers will be completely determined (since they
will each already have three neighbors with different labels).
It is not hard to check that in the central layer, the hexagon
surrounding the central copy of $D_4$ must be colored either
$bcdbcd$ or $bcbdcd$, up to a permutation of the colors,
because no other choices will extend consistently to the
neighboring layers. This determines all the labels in the two
diagrams shown in \fullref{fig:fcc}, except for the six points
labeled $a_1$. For those points, one can check that if their
labels agree with the label assigned to the central sphere,
then they add one to the kissing number; otherwise they add
zero. Thus, they should all be labeled $a$ to get a kissing
number of $126$.

Accordingly, we get two local configurations of tight packings, as
shown in \fullref{fig:fcc}.  The code $\sCl{7}{126a}$ is the $E_7$ root
system, but $\sCl{7}{126b}$ is not the kissing configuration\break 
of any of the tight packings in $\R^7$ described by Conway and
Sloane~\cite{Allthebest}.

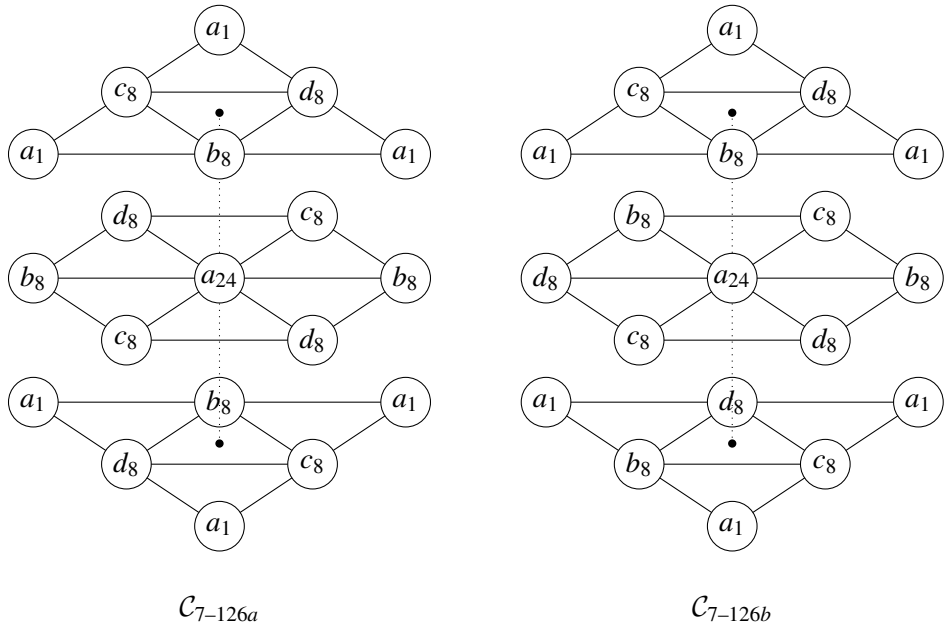
\begin{figure}[ht!]
\centering
\begin{tikzpicture}[scale=1.1]

\draw (-3*0.75,0.75*-2)--(0*0.75,0.75*-2);
\draw (-3*0.75,0.75*-2)--(-1.5*0.75,0.75*-3);
\draw (0*0.75,0.75*-2)--(-1.5*0.75,0.75*-3);
\draw (1.5*0.75,0.75*-3)--(-1.5*0.75,0.75*-3);
\draw (0*0.75,0.75*-4)--(-1.5*0.75,0.75*-3);
\draw (0*0.75,0.75*-4)--(1.5*0.75,0.75*-3);
\draw (3*0.75,0.75*-2)--(0*0.75,0.75*-2);
\draw (3*0.75,0.75*-2)--(1.5*0.75,0.75*-3);
\draw (0*0.75,0.75*-2)--(1.5*0.75,0.75*-3);

\fill [white] (-1.5*0.75,0.75*-3) circle (0.3);
\fill [white] (1.5*0.75,0.75*-3) circle (0.3);
\fill [white] (-3*0.75,0.75*-2) circle (0.3);
\fill [white] (0*0.75,0.75*-2) circle (0.3);
\fill [white] (3*0.75,0.75*-2) circle (0.3);
\fill [white] (0*0.75,0.75*-4) circle (0.3);

\draw (-1.5*0.75,0.75*-3) circle (0.3);
\draw (1.5*0.75,0.75*-3) circle (0.3);
\draw (-3*0.75,0.75*-2) circle (0.3);
\draw (0*0.75,0.75*-2) circle (0.3);
\draw (3*0.75,0.75*-2) circle (0.3);
\draw (0*0.75,0.75*-4) circle (0.3);

\draw (-1.5*0.75,0.75*-3) node {$d_8$};
\draw (1.5*0.75,0.75*-3) node {$c_8$};
\draw (-3*0.75,0.75*-2) node {$a_1$};
\draw (0*0.75,0.75*-2) node {$b_8$};
\draw (3*0.75,0.75*-2) node {$a_1$};
\draw (0*0.75,0.75*-4) node {$a_1$};

\draw[dotted] (0*0.75,0.75*0) -- (0*0.75,0.75*2.667);
\draw[dotted] (0*0.75,0.75*0) -- (0*0.75,0.75*-2.667);

\draw (0*0.75,0.75*0)--(-3*0.75,0.75*0);
\draw (0*0.75,0.75*0)--(3*0.75,0.75*0);
\draw (0*0.75,0.75*0)--(-1.5*0.75,0.75*1);
\draw (0*0.75,0.75*0)--(1.5*0.75,0.75*1);
\draw (0*0.75,0.75*0)--(-1.5*0.75,0.75*-1);
\draw (0*0.75,0.75*0)--(1.5*0.75,0.75*-1);
\draw (-3*0.75,0.75*0)--(-1.5*0.75,0.75*1);
\draw (-1.5*0.75,0.75*1)--(1.5*0.75,0.75*1);
\draw (1.5*0.75,0.75*1)--(3*0.75,0.75*0);
\draw (3*0.75,0.75*0)--(1.5*0.75,0.75*-1);
\draw (1.5*0.75,0.75*-1)--(-1.5*0.75,0.75*-1);
\draw (-1.5*0.75,0.75*-1)--(-3*0.75,0.75*0);

\fill [white] (-3*0.75,0.75*0) circle (0.3);
\fill [white] (0*0.75,0.75*0) circle (0.3);
\fill [white] (3*0.75,0.75*0) circle (0.3);
\fill [white] (-1.5*0.75,0.75*1) circle (0.3);
\fill [white] (1.5*0.75,0.75*1) circle (0.3);
\fill [white] (-1.5*0.75,0.75*-1) circle (0.3);
\fill [white] (1.5*0.75,0.75*-1) circle (0.3);

\draw (-3*0.75,0.75*0) circle (0.3);
\draw (0*0.75,0.75*0) circle (0.3);
\draw (3*0.75,0.75*0) circle (0.3);
\draw (-1.5*0.75,0.75*1) circle (0.3);
\draw (1.5*0.75,0.75*1) circle (0.3);
\draw (-1.5*0.75,0.75*-1) circle (0.3);
\draw (1.5*0.75,0.75*-1) circle (0.3);

\draw (-3*0.75,0.75*0) node {$b_8$};
\draw (0*0.75,0.75*0) node {$a_{24}$};
\draw (3*0.75,0.75*0) node {$b_8$};
\draw (-1.5*0.75,0.75*1) node {$d_8$};
\draw (1.5*0.75,0.75*1) node {$c_8$};
\draw (-1.5*0.75,0.75*-1) node {$c_8$};
\draw (1.5*0.75,0.75*-1) node {$d_8$};

\draw (-3*0.75,0.75*2)--(0*0.75,0.75*2);
\draw (-3*0.75,0.75*2)--(-1.5*0.75,0.75*3);
\draw (0*0.75,0.75*2)--(-1.5*0.75,0.75*3);
\draw (1.5*0.75,0.75*3)--(-1.5*0.75,0.75*3);
\draw (0*0.75,0.75*4)--(-1.5*0.75,0.75*3);
\draw (0*0.75,0.75*4)--(1.5*0.75,0.75*3);
\draw (3*0.75,0.75*2)--(0*0.75,0.75*2);
\draw (3*0.75,0.75*2)--(1.5*0.75,0.75*3);
\draw (0*0.75,0.75*2)--(1.5*0.75,0.75*3);

\fill [white] (-1.5*0.75,0.75*3) circle (0.3);
\fill [white] (1.5*0.75,0.75*3) circle (0.3);
\fill [white] (-3*0.75,0.75*2) circle (0.3);
\fill [white] (0*0.75,0.75*2) circle (0.3);
\fill [white] (3*0.75,0.75*2) circle (0.3);
\fill [white] (0*0.75,0.75*4) circle (0.3);

\draw (-1.5*0.75,0.75*3) circle (0.3);
\draw (1.5*0.75,0.75*3) circle (0.3);
\draw (-3*0.75,0.75*2) circle (0.3);
\draw (0*0.75,0.75*2) circle (0.3);
\draw (3*0.75,0.75*2) circle (0.3);
\draw (0*0.75,0.75*4) circle (0.3);

\draw (-1.5*0.75,0.75*3) node {$c_8$};
\draw (1.5*0.75,0.75*3) node {$d_8$};
\draw (-3*0.75,0.75*2) node {$a_1$};
\draw (0*0.75,0.75*2) node {$b_8$};
\draw (3*0.75,0.75*2) node {$a_1$};
\draw (0*0.75,0.75*4) node {$a_1$};

\fill[black] (0*0.75,0.75*2.667) circle (0.05);
\fill[black] (0*0.75,0.75*-2.667) circle (0.05);

\draw (0,-4) node {$\sCl{7}{126a}$};

\end{tikzpicture}
\hspace{1cm}
\begin{tikzpicture}[scale=1.1]

\draw (-3*0.75,0.75*-2)--(0*0.75,0.75*-2);
\draw (-3*0.75,0.75*-2)--(-1.5*0.75,0.75*-3);
\draw (0*0.75,0.75*-2)--(-1.5*0.75,0.75*-3);
\draw (1.5*0.75,0.75*-3)--(-1.5*0.75,0.75*-3);
\draw (0*0.75,0.75*-4)--(-1.5*0.75,0.75*-3);
\draw (0*0.75,0.75*-4)--(1.5*0.75,0.75*-3);
\draw (3*0.75,0.75*-2)--(0*0.75,0.75*-2);
\draw (3*0.75,0.75*-2)--(1.5*0.75,0.75*-3);
\draw (0*0.75,0.75*-2)--(1.5*0.75,0.75*-3);

\fill [white] (-1.5*0.75,0.75*-3) circle (0.3);
\fill [white] (1.5*0.75,0.75*-3) circle (0.3);
\fill [white] (-3*0.75,0.75*-2) circle (0.3);
\fill [white] (0*0.75,0.75*-2) circle (0.3);
\fill [white] (3*0.75,0.75*-2) circle (0.3);
\fill [white] (0*0.75,0.75*-4) circle (0.3);

\draw (-1.5*0.75,0.75*-3) circle (0.3);
\draw (1.5*0.75,0.75*-3) circle (0.3);
\draw (-3*0.75,0.75*-2) circle (0.3);
\draw (0*0.75,0.75*-2) circle (0.3);
\draw (3*0.75,0.75*-2) circle (0.3);
\draw (0*0.75,0.75*-4) circle (0.3);

\draw (-1.5*0.75,0.75*-3) node {$b_8$};
\draw (1.5*0.75,0.75*-3) node {$c_8$};
\draw (-3*0.75,0.75*-2) node {$a_1$};
\draw (0*0.75,0.75*-2) node {$d_8$};
\draw (3*0.75,0.75*-2) node {$a_1$};
\draw (0*0.75,0.75*-4) node {$a_1$};

\draw[dotted] (0*0.75,0.75*0) -- (0*0.75,0.75*2.667);
\draw[dotted] (0*0.75,0.75*0) -- (0*0.75,0.75*-2.667);

\draw (0*0.75,0.75*0)--(-3*0.75,0.75*0);
\draw (0*0.75,0.75*0)--(3*0.75,0.75*0);
\draw (0*0.75,0.75*0)--(-1.5*0.75,0.75*1);
\draw (0*0.75,0.75*0)--(1.5*0.75,0.75*1);
\draw (0*0.75,0.75*0)--(-1.5*0.75,0.75*-1);
\draw (0*0.75,0.75*0)--(1.5*0.75,0.75*-1);
\draw (-3*0.75,0.75*0)--(-1.5*0.75,0.75*1);
\draw (-1.5*0.75,0.75*1)--(1.5*0.75,0.75*1);
\draw (1.5*0.75,0.75*1)--(3*0.75,0.75*0);
\draw (3*0.75,0.75*0)--(1.5*0.75,0.75*-1);
\draw (1.5*0.75,0.75*-1)--(-1.5*0.75,0.75*-1);
\draw (-1.5*0.75,0.75*-1)--(-3*0.75,0.75*0);

\fill [white] (-3*0.75,0.75*0) circle (0.3);
\fill [white] (0*0.75,0.75*0) circle (0.3);
\fill [white] (3*0.75,0.75*0) circle (0.3);
\fill [white] (-1.5*0.75,0.75*1) circle (0.3);
\fill [white] (1.5*0.75,0.75*1) circle (0.3);
\fill [white] (-1.5*0.75,0.75*-1) circle (0.3);
\fill [white] (1.5*0.75,0.75*-1) circle (0.3);

\draw (-3*0.75,0.75*0) circle (0.3);
\draw (0*0.75,0.75*0) circle (0.3);
\draw (3*0.75,0.75*0) circle (0.3);
\draw (-1.5*0.75,0.75*1) circle (0.3);
\draw (1.5*0.75,0.75*1) circle (0.3);
\draw (-1.5*0.75,0.75*-1) circle (0.3);
\draw (1.5*0.75,0.75*-1) circle (0.3);

\draw (-3*0.75,0.75*0) node {$d_8$};
\draw (0*0.75,0.75*0) node {$a_{24}$};
\draw (3*0.75,0.75*0) node {$b_8$};
\draw (-1.5*0.75,0.75*1) node {$b_8$};
\draw (1.5*0.75,0.75*1) node {$c_8$};
\draw (-1.5*0.75,0.75*-1) node {$c_8$};
\draw (1.5*0.75,0.75*-1) node {$d_8$};

\draw (-3*0.75,0.75*2)--(0*0.75,0.75*2);
\draw (-3*0.75,0.75*2)--(-1.5*0.75,0.75*3);
\draw (0*0.75,0.75*2)--(-1.5*0.75,0.75*3);
\draw (1.5*0.75,0.75*3)--(-1.5*0.75,0.75*3);
\draw (0*0.75,0.75*4)--(-1.5*0.75,0.75*3);
\draw (0*0.75,0.75*4)--(1.5*0.75,0.75*3);
\draw (3*0.75,0.75*2)--(0*0.75,0.75*2);
\draw (3*0.75,0.75*2)--(1.5*0.75,0.75*3);
\draw (0*0.75,0.75*2)--(1.5*0.75,0.75*3);

\fill [white] (-1.5*0.75,0.75*3) circle (0.3);
\fill [white] (1.5*0.75,0.75*3) circle (0.3);
\fill [white] (-3*0.75,0.75*2) circle (0.3);
\fill [white] (0*0.75,0.75*2) circle (0.3);
\fill [white] (3*0.75,0.75*2) circle (0.3);
\fill [white] (0*0.75,0.75*4) circle (0.3);

\draw (-1.5*0.75,0.75*3) circle (0.3);
\draw (1.5*0.75,0.75*3) circle (0.3);
\draw (-3*0.75,0.75*2) circle (0.3);
\draw (0*0.75,0.75*2) circle (0.3);
\draw (3*0.75,0.75*2) circle (0.3);
\draw (0*0.75,0.75*4) circle (0.3);

\draw (-1.5*0.75,0.75*3) node {$c_8$};
\draw (1.5*0.75,0.75*3) node {$d_8$};
\draw (-3*0.75,0.75*2) node {$a_1$};
\draw (0*0.75,0.75*2) node {$b_8$};
\draw (3*0.75,0.75*2) node {$a_1$};
\draw (0*0.75,0.75*4) node {$a_1$};

\fill[black] (0*0.75,0.75*2.667) circle (0.05);
\fill[black] (0*0.75,0.75*-2.667) circle (0.05);

\draw (0,-4) node {$\sCl{7}{126b}$};
\end{tikzpicture}
\vspace{-0.25cm} 
\caption{Kissing configurations in $\R^7$
based on the face-centered cubic packing} \label{fig:fcc}
\end{figure}

The two other conjecturally optimal kissing arrangements in $S^6$ come
from the hexagonal close-packing via a similar argument and are shown
in \fullref{fig:hcp}.  The code $\sCl{7}{126c}$ is the kissing
configuration of the packing $\Lambda_7^2$ described in
\cite{Allthebest}, while $\sCl{7}{126d}$ does\break
not occur in any of the tight packings analyzed in that paper.

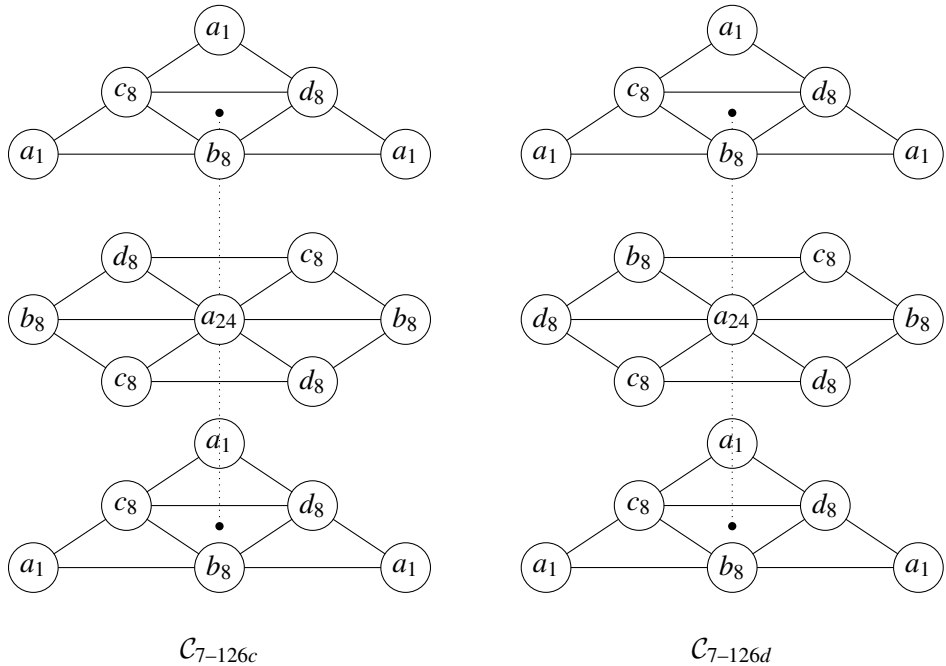
\begin{figure}[ht!]
\centering
\begin{tikzpicture}[scale=1.1]

\draw (-3*0.75,0.75*-4)--(0*0.75,0.75*-4);
\draw (-3*0.75,0.75*-4)--(-1.5*0.75,0.75*-3);
\draw (0*0.75,0.75*-4)--(-1.5*0.75,0.75*-3);
\draw (1.5*0.75,0.75*-3)--(-1.5*0.75,0.75*-3);
\draw (0*0.75,0.75*-2)--(-1.5*0.75,0.75*-3);
\draw (0*0.75,0.75*-2)--(1.5*0.75,0.75*-3);
\draw (3*0.75,0.75*-4)--(0*0.75,0.75*-4);
\draw (3*0.75,0.75*-4)--(1.5*0.75,0.75*-3);
\draw (0*0.75,0.75*-4)--(1.5*0.75,0.75*-3);

\fill [white] (-1.5*0.75,0.75*-3) circle (0.3);
\fill [white] (1.5*0.75,0.75*-3) circle (0.3);
\fill [white] (-3*0.75,0.75*-4) circle (0.3);
\fill [white] (0*0.75,0.75*-4) circle (0.3);
\fill [white] (3*0.75,0.75*-4) circle (0.3);
\fill [white] (0*0.75,0.75*-2) circle (0.3);

\draw (-1.5*0.75,0.75*-3) circle (0.3);
\draw (1.5*0.75,0.75*-3) circle (0.3);
\draw (-3*0.75,0.75*-4) circle (0.3);
\draw (0*0.75,0.75*-4) circle (0.3);
\draw (3*0.75,0.75*-4) circle (0.3);
\draw (0*0.75,0.75*-2) circle (0.3);

\draw (-1.5*0.75,0.75*-3) node {$c_8$};
\draw (1.5*0.75,0.75*-3) node {$d_8$};
\draw (-3*0.75,0.75*-4) node {$a_1$};
\draw (0*0.75,0.75*-4) node {$b_8$};
\draw (3*0.75,0.75*-4) node {$a_1$};
\draw (0*0.75,0.75*-2) node {$a_1$};

\draw[dotted] (0*0.75,0.75*0) -- (0*0.75,0.75*3.333);
\draw[dotted] (0*0.75,0.75*0) -- (0*0.75,0.75*-3.333);

\draw (0*0.75,0.75*0)--(-3*0.75,0.75*0);
\draw (0*0.75,0.75*0)--(3*0.75,0.75*0);
\draw (0*0.75,0.75*0)--(-1.5*0.75,0.75*1);
\draw (0*0.75,0.75*0)--(1.5*0.75,0.75*1);
\draw (0*0.75,0.75*0)--(-1.5*0.75,0.75*-1);
\draw (0*0.75,0.75*0)--(1.5*0.75,0.75*-1);
\draw (-3*0.75,0.75*0)--(-1.5*0.75,0.75*1);
\draw (-1.5*0.75,0.75*1)--(1.5*0.75,0.75*1);
\draw (1.5*0.75,0.75*1)--(3*0.75,0.75*0);
\draw (3*0.75,0.75*0)--(1.5*0.75,0.75*-1);
\draw (1.5*0.75,0.75*-1)--(-1.5*0.75,0.75*-1);
\draw (-1.5*0.75,0.75*-1)--(-3*0.75,0.75*0);

\fill [white] (-3*0.75,0.75*0) circle (0.3);
\fill [white] (0*0.75,0.75*0) circle (0.3);
\fill [white] (3*0.75,0.75*0) circle (0.3);
\fill [white] (-1.5*0.75,0.75*1) circle (0.3);
\fill [white] (1.5*0.75,0.75*1) circle (0.3);
\fill [white] (-1.5*0.75,0.75*-1) circle (0.3);
\fill [white] (1.5*0.75,0.75*-1) circle (0.3);

\draw (-3*0.75,0.75*0) circle (0.3);
\draw (0*0.75,0.75*0) circle (0.3);
\draw (3*0.75,0.75*0) circle (0.3);
\draw (-1.5*0.75,0.75*1) circle (0.3);
\draw (1.5*0.75,0.75*1) circle (0.3);
\draw (-1.5*0.75,0.75*-1) circle (0.3);
\draw (1.5*0.75,0.75*-1) circle (0.3);

\draw (-3*0.75,0.75*0) node {$b_8$};
\draw (0*0.75,0.75*0) node {$a_{24}$};
\draw (3*0.75,0.75*0) node {$b_8$};
\draw (-1.5*0.75,0.75*1) node {$d_8$};
\draw (1.5*0.75,0.75*1) node {$c_8$};
\draw (-1.5*0.75,0.75*-1) node {$c_8$};
\draw (1.5*0.75,0.75*-1) node {$d_8$};

\draw (-3*0.75,0.75*2.667)--(0*0.75,0.75*2.667);
\draw (-3*0.75,0.75*2.667)--(-1.5*0.75,0.75*3.667);
\draw (0*0.75,0.75*2.667)--(-1.5*0.75,0.75*3.667);
\draw (1.5*0.75,0.75*3.667)--(-1.5*0.75,0.75*3.667);
\draw (0*0.75,0.75*4.667)--(-1.5*0.75,0.75*3.667);
\draw (0*0.75,0.75*4.667)--(1.5*0.75,0.75*3.667);
\draw (3*0.75,0.75*2.667)--(0*0.75,0.75*2.667);
\draw (3*0.75,0.75*2.667)--(1.5*0.75,0.75*3.667);
\draw (0*0.75,0.75*2.667)--(1.5*0.75,0.75*3.667);

\fill [white] (-1.5*0.75,0.75*3.667) circle (0.3);
\fill [white] (1.5*0.75,0.75*3.667) circle (0.3);
\fill [white] (-3*0.75,0.75*2.667) circle (0.3);
\fill [white] (0*0.75,0.75*2.667) circle (0.3);
\fill [white] (3*0.75,0.75*2.667) circle (0.3);
\fill [white] (0*0.75,0.75*4.667) circle (0.3);

\draw (-1.5*0.75,0.75*3.667) circle (0.3);
\draw (1.5*0.75,0.75*3.667) circle (0.3);
\draw (-3*0.75,0.75*2.667) circle (0.3);
\draw (0*0.75,0.75*2.667) circle (0.3);
\draw (3*0.75,0.75*2.667) circle (0.3);
\draw (0*0.75,0.75*4.667) circle (0.3);

\draw (-1.5*0.75,0.75*3.667) node {$c_8$};
\draw (1.5*0.75,0.75*3.667) node {$d_8$};
\draw (-3*0.75,0.75*2.667) node {$a_1$};
\draw (0*0.75,0.75*2.667) node {$b_8$};
\draw (3*0.75,0.75*2.667) node {$a_1$};
\draw (0*0.75,0.75*4.667) node {$a_1$};

\fill[black] (0*0.75,0.75*3.333) circle (0.05);
\fill[black] (0*0.75,0.75*-3.333) circle (0.05);

\draw (0,-4) node {$\sCl{7}{126c}$};

\end{tikzpicture}
\hspace{1cm}
\begin{tikzpicture}[scale=1.1]

\draw (-3*0.75,0.75*-4)--(0*0.75,0.75*-4);
\draw (-3*0.75,0.75*-4)--(-1.5*0.75,0.75*-3);
\draw (0*0.75,0.75*-4)--(-1.5*0.75,0.75*-3);
\draw (1.5*0.75,0.75*-3)--(-1.5*0.75,0.75*-3);
\draw (0*0.75,0.75*-2)--(-1.5*0.75,0.75*-3);
\draw (0*0.75,0.75*-2)--(1.5*0.75,0.75*-3);
\draw (3*0.75,0.75*-4)--(0*0.75,0.75*-4);
\draw (3*0.75,0.75*-4)--(1.5*0.75,0.75*-3);
\draw (0*0.75,0.75*-4)--(1.5*0.75,0.75*-3);

\fill [white] (-1.5*0.75,0.75*-3) circle (0.3);
\fill [white] (1.5*0.75,0.75*-3) circle (0.3);
\fill [white] (-3*0.75,0.75*-4) circle (0.3);
\fill [white] (0*0.75,0.75*-4) circle (0.3);
\fill [white] (3*0.75,0.75*-4) circle (0.3);
\fill [white] (0*0.75,0.75*-2) circle (0.3);

\draw (-1.5*0.75,0.75*-3) circle (0.3);
\draw (1.5*0.75,0.75*-3) circle (0.3);
\draw (-3*0.75,0.75*-4) circle (0.3);
\draw (0*0.75,0.75*-4) circle (0.3);
\draw (3*0.75,0.75*-4) circle (0.3);
\draw (0*0.75,0.75*-2) circle (0.3);

\draw (-1.5*0.75,0.75*-3) node {$c_8$};
\draw (1.5*0.75,0.75*-3) node {$d_8$};
\draw (-3*0.75,0.75*-4) node {$a_1$};
\draw (0*0.75,0.75*-4) node {$b_8$};
\draw (3*0.75,0.75*-4) node {$a_1$};
\draw (0*0.75,0.75*-2) node {$a_1$};

\draw[dotted] (0*0.75,0.75*0) -- (0*0.75,0.75*3.333);
\draw[dotted] (0*0.75,0.75*0) -- (0*0.75,0.75*-3.333);

\draw (0*0.75,0.75*0)--(-3*0.75,0.75*0);
\draw (0*0.75,0.75*0)--(3*0.75,0.75*0);
\draw (0*0.75,0.75*0)--(-1.5*0.75,0.75*1);
\draw (0*0.75,0.75*0)--(1.5*0.75,0.75*1);
\draw (0*0.75,0.75*0)--(-1.5*0.75,0.75*-1);
\draw (0*0.75,0.75*0)--(1.5*0.75,0.75*-1);
\draw (-3*0.75,0.75*0)--(-1.5*0.75,0.75*1);
\draw (-1.5*0.75,0.75*1)--(1.5*0.75,0.75*1);
\draw (1.5*0.75,0.75*1)--(3*0.75,0.75*0);
\draw (3*0.75,0.75*0)--(1.5*0.75,0.75*-1);
\draw (1.5*0.75,0.75*-1)--(-1.5*0.75,0.75*-1);
\draw (-1.5*0.75,0.75*-1)--(-3*0.75,0.75*0);

\fill [white] (-3*0.75,0.75*0) circle (0.3);
\fill [white] (0*0.75,0.75*0) circle (0.3);
\fill [white] (3*0.75,0.75*0) circle (0.3);
\fill [white] (-1.5*0.75,0.75*1) circle (0.3);
\fill [white] (1.5*0.75,0.75*1) circle (0.3);
\fill [white] (-1.5*0.75,0.75*-1) circle (0.3);
\fill [white] (1.5*0.75,0.75*-1) circle (0.3);

\draw (-3*0.75,0.75*0) circle (0.3);
\draw (0*0.75,0.75*0) circle (0.3);
\draw (3*0.75,0.75*0) circle (0.3);
\draw (-1.5*0.75,0.75*1) circle (0.3);
\draw (1.5*0.75,0.75*1) circle (0.3);
\draw (-1.5*0.75,0.75*-1) circle (0.3);
\draw (1.5*0.75,0.75*-1) circle (0.3);

\draw (-3*0.75,0.75*0) node {$d_8$};
\draw (0*0.75,0.75*0) node {$a_{24}$};
\draw (3*0.75,0.75*0) node {$b_8$};
\draw (-1.5*0.75,0.75*1) node {$b_8$};
\draw (1.5*0.75,0.75*1) node {$c_8$};
\draw (-1.5*0.75,0.75*-1) node {$c_8$};
\draw (1.5*0.75,0.75*-1) node {$d_8$};

\draw (-3*0.75,0.75*2.667)--(0*0.75,0.75*2.667);
\draw (-3*0.75,0.75*2.667)--(-1.5*0.75,0.75*3.667);
\draw (0*0.75,0.75*2.667)--(-1.5*0.75,0.75*3.667);
\draw (1.5*0.75,0.75*3.667)--(-1.5*0.75,0.75*3.667);
\draw (0*0.75,0.75*4.667)--(-1.5*0.75,0.75*3.667);
\draw (0*0.75,0.75*4.667)--(1.5*0.75,0.75*3.667);
\draw (3*0.75,0.75*2.667)--(0*0.75,0.75*2.667);
\draw (3*0.75,0.75*2.667)--(1.5*0.75,0.75*3.667);
\draw (0*0.75,0.75*2.667)--(1.5*0.75,0.75*3.667);

\fill [white] (-1.5*0.75,0.75*3.667) circle (0.3);
\fill [white] (1.5*0.75,0.75*3.667) circle (0.3);
\fill [white] (-3*0.75,0.75*2.667) circle (0.3);
\fill [white] (0*0.75,0.75*2.667) circle (0.3);
\fill [white] (3*0.75,0.75*2.667) circle (0.3);
\fill [white] (0*0.75,0.75*4.667) circle (0.3);

\draw (-1.5*0.75,0.75*3.667) circle (0.3);
\draw (1.5*0.75,0.75*3.667) circle (0.3);
\draw (-3*0.75,0.75*2.667) circle (0.3);
\draw (0*0.75,0.75*2.667) circle (0.3);
\draw (3*0.75,0.75*2.667) circle (0.3);
\draw (0*0.75,0.75*4.667) circle (0.3);

\draw (-1.5*0.75,0.75*3.667) node {$c_8$};
\draw (1.5*0.75,0.75*3.667) node {$d_8$};
\draw (-3*0.75,0.75*2.667) node {$a_1$};
\draw (0*0.75,0.75*2.667) node {$b_8$};
\draw (3*0.75,0.75*2.667) node {$a_1$};
\draw (0*0.75,0.75*4.667) node {$a_1$};

\fill[black] (0*0.75,0.75*3.333) circle (0.05);
\fill[black] (0*0.75,0.75*-3.333) circle (0.05);

\draw (0,-4) node {$\sCl{7}{126d}$};

\end{tikzpicture}
\vspace{-0.25cm} 
\caption{Kissing configurations in $\R^7$ based
on the hexagonal close-packing} \label{fig:hcp}
\end{figure}

These four spherical codes can again be distinguished by the sizes of
their auto\nobreak{}mor\nobreak{}phism\break
groups, which are $2903040$, $46080$, $103680$ and $3840$,
respectively.

\section{Nine through twelve dimensions} \label{section:9d}

Up through eight dimensions, we are quite confident that the known
kissing configurations are optimal and reasonably confident that the
list we have provided is complete.  However, the situation in nine or
more dimensions is very different. For example, in $\R^9$ the highest
kissing number achieved by any lattice is $272$ (by a theorem of
Watson~\cite{W}), while the best kissing number known is $306$.
Furthermore, there is little\break
reason to believe that the usual constructions based on
error-correcting codes (see Conway and
Sloane~\cite[Chapter~5]{SPLAG}) will prove adequate in high
dimensions.  Even in\break
$\R^{10}$, computer searches have led to intriguing new kissing
configurations (see Elser and Gravel~\cite{EG10}), although so
far they have not improved on the known records.

The laminated lattice $\Lambda_9$ achieves kissing number $272$, and it
is the only lattice that does. Here, we prove two properties of this
kissing configuration: it is not locally jammed, but it is the largest
possible kissing configuration in $\R^9$ that contains the $E_8$ root
system as a cross section.

To describe this code, we first recall the structure of the root
lattice $E_8$, which is used to produce $\Lambda_9$. It the union of
two translates of $D_8$, namely
$$
E_8 = D_8 \cup \big(D_8 + (1/2, \dots, 1/2) \big).
$$
The minimal vectors of $D_8$ are the $112$ vectors of the form
$\pm e_i \pm e_j$ with $i \ne j$\break
(where $e_i$ is the $i$th standard unit vector), and the $E_8$
kissing configuration consists of these vectors as well as the
$128$ vectors of the form $(\pm 1/2, \dots, \pm 1/2)$ with an
even number of minus signs.

The vector $v = (1,0,0,\dots,0)$ is a deep hole of $E_8$, and the
lattice $\Lambda_9$ is generated\break{} 
by $E_8 \times \{0\}$ and $v \times \{1\}$ in $\R^9$. Its
kissing configuration consists of $272$ vectors,
and\break
is the union of the kissing configuration $S$ of $D_9$ (which
contains $144$ vectors) and\break
the set $T$ of $128$ vectors of the form $(\pm 1/2, \dots, \pm
1/2, 0)$ with an even number of minus signs.

By \fullref{cor:rootjammed}, the subconfiguration $S$ is
itself jammed, so it is futile to move those points of the
code, but we will show that the points in $T$ are not even
locally jammed.  The reason is that they have no neighbors
outside of $E_8 \times \{0\}$, so they are free to move
orthogonally to $E_8$.

More formally, given $(x,0) \in T$ with $x \in E_8
\setminus D_8$,
we move it to $\big(x\cos \theta, \sqrt{2} \sin
\theta\big)$, where $-\pi/2 \le \theta \le \pi/2$.  For every
point $y \in E_8$ with $y \ne x$, we have $\langle x,y \rangle
\le 1$ and hence
$$
\big\langle \big(x \cos \theta, \sqrt{2} \sin \theta\big) ,
(y,0) \big\rangle = \langle x,y \rangle \cos\theta  \le 1.
$$
Therefore, the only points $\big(x\cos \theta, \sqrt{2} \sin
\theta\big)$ can come too close to are those in\break
$D_9 \setminus D_8$, that is, those of the form $\pm e_i \pm
e_9$ with $1 \le i \le 8$.  We do indeed run into\break
problems from these points when $|\theta|$ is too large, but
not when $|\theta|$ is small. The relevant inner product is
$$
\big\langle(x \cos \theta, \sqrt{2} \sin \theta\big),
\pm e_i \pm e_9 \big \rangle
= \pm \frac{1}{2} \cos \theta \pm \sqrt{2} \sin \theta.
$$
When $|\theta|$ is sufficiently small, the right side is
clearly bounded by $1$.  Specifically, we obtain a bound of $1$
whenever
$$
|\theta| \le \alpha \stackrel{\textrm{def}}{=}
\tan^{-1} \frac{\sqrt{5}-\sqrt{2}}{2} \approx 0.124  \pi.
$$
Thus, the point $(x,0)$ is not locally jammed.

Furthermore, we can even simultaneously move all the points of
$T$ by moving $(x,0)$ to\break
$(x \cos \theta_x, \sqrt{2} \sin \theta_x)$, where $-\alpha \le
\theta_x \le \alpha$ and $\theta_x$ has sign $(-1)^{k_x/2}$,
with $k_x$ being\break
the number of negative coordinates of $x$. All we need to check
is whether these points come too close to each other.  The
inner product between two of them is simply
$$\big\langle \big(x \cos \theta_x,\sqrt{2} \sin \theta_x\big),
\big(y \cos \theta_y, \sqrt{2} \sin \theta_y\big) \big\rangle {=}
\langle x,y \rangle
(\cos \theta_x)(\cos \theta_y) {+} 2 (\sin \theta_x)(\sin \theta_y).$$
If $\langle x,y \rangle = 1$, then $x$ and $y$ have differing signs in
two coordinates (that is, $k_x$ and\break
$k_y$ differ by $2$). In that case $\theta_x$ and $\theta_y$
have opposite signs, so
$$
(\cos \theta_x)(\cos \theta_y) + 2 (\sin \theta_x)(\sin \theta_y)
\le (\cos \theta_x) (\cos \theta_y) \le 1.
$$
On the other hand, if $\langle x, y \rangle \le 0$, then the inner
product is at most $2 (\sin \theta_x)(\sin \theta_y)$, which is again
at most $1$ (because $|\theta_x|$ and $|\theta_y|$ are less than
$\pi/4$).

\begin{proposition} \label{prop:optikiss}
No kissing configuration in $\R^9$ that contains the $E_8$ root system
as a cross section can have more than $272$ points.
\end{proposition}

\begin{proof}
Suppose $\mathcal{C}$ is a kissing configuration that contains the
points $(x,0)$ for $x \in E_8$ with $|x|^2=2$.  The remaining points
must be of the form $(y \cos \theta, \sqrt{2} \sin \theta)$ with $y \in
\R^8$ satisfying $|y|^2=2$ and $-\pi/2 \le \theta \le \pi/2$.

Given any such point $y$, there exists a minimal vector $x \in
E_8$ such that $\langle x,y \rangle \ge \sqrt{2}$.  This claim
amounts to knowing the depth of the deep holes in the $E_8$
root system, which correspond to the cross-polytope facets of
its convex hull.  (See, for example, Conway and Sloane~\cite{cellstruct}.)  It
follows that $\theta=0$ is impossible, since $(y,0)$ would come
too\break
close to the point $(x,0)$.  In fact, because the points
$(x,0)$ and $(y \cos \theta, \sqrt{2} \sin \theta)$ in
$\mathcal{C}$ must have inner product at most $1$, we must have
$\sqrt{2} \cos \theta \le 1$. Thus, $|\kern-0.1em\sin \theta| 
\ge 1/\sqrt{2}$. Without loss of generality, we will focus on
the points with $\sin \theta \ge 1/\sqrt{2}$.

Given two distinct such points $(y \cos \theta, \sqrt{2} \sin \theta)$
and $(z \cos \varphi, \sqrt{2} \sin \varphi)$, their inner product is
again at most $1$.  Therefore,
$$
1 \ge \langle y,z \rangle \cos
\theta \cos \varphi + 2\sin \theta \sin \varphi \ge \langle y,z \rangle \cos
\theta \cos \varphi
+ 1,
$$
from which it follows that $\langle y,z \rangle \le 0$.  (Note that
$\cos \theta$ and $\cos \varphi$ cannot vanish, since then the sine
term would be too large.)

There are no more than $16$ vectors in $S^7$ for which all inner
products between distinct vectors are nonpositive (see
B{\"o}r{\"o}czky~\cite[Theorem~6.2.1]{B}).  Thus, we have shown that
$\mathcal{C}$ contains at most
$16$ additional vectors in each hemisphere, for a total of at most
$240+2\cdot 16 = 272$, as desired.
\end{proof}

As a consequence, the best kissing configuration in $\R^9$ cannot
contain the best one\break
in $\R^8$ (namely, $E_8$) as a cross section. Note also that in
any $272$--point kissing configuration containing $E_8$,
equality holds throughout the proof of \fullref{prop:optikiss},
so there are only finitely many such configurations, each
consisting of $E_8$ with cross polytopes sitting above and
below some of its deep holes. However, the relative position\break
of these cross polytopes may vary.  For example, in $\Lambda_9$ one of
the cross polytopes consists of the points $\pm e_i - e_9$ with $1 \le
i \le 8$.  If $H$ is any $4 \times 4$ Hadamard matrix, we could replace
this cross polytope with the points $\pm v_i - e_9$, where
$v_1,\dots,v_8$ are\break 
the rows of the matrix
$$
\begin{bmatrix}
H/2 & 0\\
0 & H/2
\end{bmatrix}.
$$
The resulting code is genuinely different, because it contains two
points with inner product $-3/2$, while only $\pm 2$, $\pm 1$, $\pm
1/2$ and $0$ occur in the $\Lambda_9$ configuration.

\fullref{9to12} lists the best kissing numbers known in
dimensions $9$ through $12$ (see Conway and
Sloane~\cite[pages~139--140]{SPLAG}).  The configurations are
constructed using constant weight binary codes, and in fact the
$E_7$ and $E_8$ root systems can also be constructed in this
way. Every binary code $\sB$ of block length $n$, size $N$,
constant weight $4$ and minimal distance $4$ yields a periodic
packing in $\R^n$, namely all the vectors in $\Z^n$ that reduce
to codewords in $\sB$ modulo $2$. The vectors of norm $4$ in
this packing form a kissing configuration of size $2n + 16 N$,
consisting of the points $\pm 2 e_i$ together with signed
codewords from $\sB$ (that is, vectors with arbitrary $\pm 1$
entries in the support of a\break{} 
codeword).

\begin{table}[ht!]
\centering
\begin{tabular}{ccc}
Dimension & Best known kissing number & Packing \\ \hline
9 & 306 & $P_{9a}$ \\
10 & 500 & $P_{10b}$ \\
11 & 582 & $P_{11c}$ \\
12 & 840 & $P_{12a}$
\end{tabular}
\caption{The best kissing numbers known in dimensions $9$
through $12$}\label{9to12} \vspace{-0.3cm}
\end{table}

For $n=9$, $10$, $11$ and $12$ one can achieve $N=18$, $30$, $35$ and
$51$, respectively.  These codes are unique up to isomorphism when
$n=9$ or $10$, there are $11$ of them for $n=11$, and there are $17$ of
them for $n=12$ (see~\cite[Table~I]{O10}; the $n=9$ case was proved
by \"Osterg\r{a}rd~\cite{O10}, the $n=10$ case by Barrau~\cite{B08},
and the $n=11$ and $12$ cases by Best~\cite{B77,B78,B80}). Data files
giving coordinates for these codes are available in the supplementary
information (see \fullref{appendix:data}).

For $n \ge 11$ there are multiple codes,  and each of these binary
codes yields a distinct kissing configuration. To see why, first
observe that each contains a cross polytope consisting of the vectors
of the form $\pm 2 e_i$.  This cross polytope is uniquely distinguished
by the property of having large valencies. Specifically, when $n=11$ it
consists of exactly the points with at least $88$ neighbors (that is,
points at inner product $2$), and when $n=12$ it consists of those with
at least $136$ neighbors. These distinguished cross polytopes must
correspond under any isomorphism between two configurations, which must
therefore be a signed permutation of the coordinates.  However, it
would then yield an isomorphism of the underlying binary codes.

In nine through twelve dimensions, these kissing configurations are all
jammed, due to the following proposition, whose hypotheses can speedily
be checked by a computer calculation:

\begin{proposition} \label{proposition:constantweight}
Let $\sB$ be a constant weight code of block length $n$, weight $4$
and minimal distance $4$.  If $\sB$ has the following two properties,
then the corresponding kissing configuration in $\R^n$ is
infinitesimally jammed:
\begin{enumerate}
\item \label{prop:one} For $i < j$, there is a codeword $x
    =(x_1,\dots,x_n) \in \sB$ such that $x_i=x_j=1$.

\item \label{prop:two} For every $x \in \sB$ and $i$ such that $x_i
    = 0$, there exists $y \in \sB$ such that $y_i=1$ and $d(x,y) =
    4$ (that is, the supports of $x$ and $y$ overlap in exactly two
    coordinates).
\end{enumerate}
\end{proposition}

\begin{proof}
Let $\sC$ be the kissing configuration obtained from $\sB$.
First, observe that for $i \neq j$, the points $2e_i$ and
$2e_j$ are part of a scaled $D_4$ root system embedded in
$\sC$.   Specifically, by \eqref{prop:one}, there exist $k$ and
$\ell$ such that $\pm e_i \pm e_j \pm e_k \pm e_\ell$ are all
in $\sC$; then
$$
\{\pm 2e_i, \pm 2e_j, \pm 2e_k, \pm 2e_\ell, \pm e_i \pm e_j \pm e_k \pm e_\ell\}
$$
is a scaled $D_4$ root system.  Therefore, by
\fullref{lemma:reducedim} and \fullref{prop:D4jammed}, the
inner products in $D_4$ and thus in the cross polytope $\{ \pm 2 e_i :
1 \le i \le n\}$ cannot change in any infinitesimal deformation.  Since
the cross polytope spans $\R^n$, we can assume that the points in it
are fixed.\pagebreak

Let
$$
x = s_i e_i + s_j e_j + s_k e_k + s_\ell e_{\ell}
\in \sC
$$
be any of the remaining vectors, with $s_i,s_j,s_k,s_\ell \in
\{\pm 1\}$, and let $y$ be the infinitesimal perturbation of
$x$. Then $x$ and $2e_i$ generate a scaled copy of the $A_2$
root system in $\sC$, which is also infinitesimally jammed.
Therefore the inner product between $x$ and $2 e_i$ does not
change to first order, so $y_i = 0$. Similarly $y_j = y_k =
y_{\ell} = 0$.

Now let $p \not \in \{i,j,k,\ell\}$. By \eqref{prop:two}, there
exists $b \in \sB$ with $b_p = 1$ and such that $b$ overlaps
the support of $x$ in exactly two positions. Assume without
loss of generality that $b$ is supported in positions $i$, $j$,
$p$ and $q$.  Then $\pm e_i \pm e_j \pm e_p  \pm e_q$ is in
$\sC$ for\break
all combinations of signs. The four points $x$, $z=-s_i e_i -
s_je_j - e_p - e_q$, $2e_p$ and $2e_q$ span a scaled copy of
the $D_4$ root system in $\sC$; indeed, they form
a $D_4$ Dynkin diagram: 
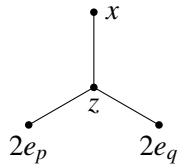
\begin{center}
\vspace{0.075cm} 
\begin{tikzpicture}
\fill (0,0) circle (0.05);
\draw (0,0) node[below] {$z$};
\fill (-0.866025403,-0.5) circle (0.05);
\draw (-0.866025403,-0.5) node[below] {$2e_p$};
\fill (0.866025403,-0.5) circle (0.05);
\draw (0.866025403,-0.5) node[below] {$2e_q$};
\fill (0,1) circle (0.05);
\draw (0,1) node[right] {$x$};
\draw (0,0) -- (-0.866025403,-0.5);
\draw (0,0) -- (0.866025403,-0.5);
\draw (0,0) -- (0,1);
\end{tikzpicture}
\vspace{0.075cm} 
\end{center}
Therefore, the inner product between $x$ and $2e_p$ does not
change to first order, so $y_p=0$.  It follows that $y=0$, and
thus $\sC$ is infinitesimally jammed.
\end{proof}

The situation is quite different for lattice kissing configurations.
The best one known in $\R^{10}$ is that of the lattice $\Lambda_{10}$;
it contains $336$ points and is not locally jammed (as one can check by
a calculation like that for $\Lambda_9$). The best lattice kissing
arrangement known in $\R^{11}$ comes from the lattice
$\Lambda_{11}^{\textrm{max}}$; it contains $438$ points and is also not
locally jammed.

The best lattice kissing configuration known in $\R^{12}$ is that of
the Coxeter--Todd lattice $K_{12}$, which is not a laminated lattice.
Unlike what happens in the previous three dimensions, this $756$--point
code is locally jammed and has a transitive symmetry group. However, we
will show that it is not in fact jammed.  Aside from the $A_n$ root\break
systems, the Coxeter--Todd kissing configuration is the only example
analyzed in this\break
paper that is locally jammed but is not jammed.  It is also
remarkable because the Coxeter--Todd lattice is the densest
sphere packing known in $\R^{12}$.  This phenomenon of
seemingly optimal packings with locally jammed yet unjammed
kissing configurations also occurs in three dimensions, but we
know of no other cases.

To unjam the Coxeter--Todd kissing configuration, we will make
use of its Eisenstein structure: $K_{12}$ is a
$\Z[\omega]$--module of dimension $6$, where $\omega$ is a
primitive cube root of\pagebreak\break{} 
unity.  We view $K_{12}$ as a subset of $\C^6$ with the inner
product $\langle x,y \rangle = 2\Re \langle x,y \rangle_\C$,
where $\langle z,w \rangle_\C = \smash{\sum_{j=1}^6} z_j
\overline{w}_j$ is the usual Hermitian inner product on $\C^6$.
The minimal vectors of $K_{12}$ have norm $4$ with respect to
$\langle \cdot,\cdot \rangle$.

The key property of $K_{12}$ is that for minimal vectors $x$ and $y$
satisfying $x \ne \pm \omega^j y$, even the complex inner product
$\langle x,y \rangle_\C$ is bounded \emph{in absolute value} by $1$.
This is quite unusual and does not hold for most other Eisenstein
lattices (or related types of lattices such as Hurwitz lattices), but
it can be checked directly from the list of minimal vectors in Conway
and Sloane~\cite[page~128]{SPLAG}. More conceptually, it follows from
the Eisenstein integrality of $K_{12}$ (that is, the fact that all the
complex inner products are in $\Z[\omega]$).  Specifically, if $|\langle
x,y \rangle_\C| > 1$, then a quick enumeration shows that the only
possibility is $\langle x,y \rangle_\C = \pm \omega^j (1-\omega)$, in
which case $\langle x, \pm \omega^j y \rangle_\C = 1-\omega$.  However,
a complex inner product of $1-\omega$ leads to a real inner product of
$3$, which is impossible.  Note that this property is an assertion about
the minimal distance of the $126$--point configuration in $\C\Proj^5$
obtained by taking the quotient of this code by the action of the
multiplicative group $\C^\times$, which in this case just amounts to
taking the quotient modulo the sixth roots of unity.  In fact, it follows
from Cohn and Kumar~\cite[Theorem~8.2]{Univopt} that the resulting code
in $\C\Proj^5$ is universally optimal.

The Eisenstein structure breaks up the code into $126$ hexagons
(the orbits under multiplication by powers of the sixth root of
unity $-\omega$), and to unjam it we simply rotate each of
these hexagons by arbitrary angles. If we rotate $x$ to
$e^{i\theta} x$ and $y$ to\break{} 
$e^{i\varphi} y$, then
$$
\big\langle e^{i\theta} x , e^{i \varphi}y \big\rangle_\C = e^{i(\theta-\varphi)}
\langle x,y\rangle_\C,
$$
and hence
$$
\big\langle e^{i\theta} x , e^{i \varphi}y \big\rangle
= 2\Re \big\langle e^{i\theta} x , e^{i \varphi}y \big\rangle_\C
\le 2\big|\big\langle e^{i\theta} x , e^{i \varphi}y \big\rangle_\C\big|
= 2|\langle x,y\rangle_\C| \le 2.
$$
If each hexagon is rotated by a slightly different angle,
then this inequality will be strict whenever $x$ and $y$ lie in
different hexagons.  Therefore, we get an actual unjamming.  Within
each hexagon, the minimal distance has not changed.  However, we can
follow up such a deformation by moving the elements in each of the
hexagons away from each other.  For instance, decompose the hexagon
into two equilateral triangles and move them away from each other, in a
direction orthogonal to the plane of the hexagon.  This process
increases the minimal distance within each hexagon and results in a
configuration of $756$ points with minimal angle strictly larger than
$\pi/3$.  Thus, we\break{} 
can not just unjam the $K_{12}$ kissing configuration, but
unjam it so thoroughly that no contact remains (much like in
three dimensions).

Note that it is difficult to predict results such as this based on
dimension counting.  For example, for large codes in $S^{n-1}$, the
\emph{isostatic condition} suggests that about $2(n-1)$ contacts per
particle are needed to ensure jamming (see Torquato and
Stillinger~\cite[page~2641]{TS}), because that
yields $N(n-1)$ constraints for an $N$--particle code, which is the same
as the number of degrees of freedom.  (Strictly speaking, we should
subtract $\dim O(n) = \binom{n}{2}$ from the number of degrees of
freedom, but that is negligible when $N$ is large.)  For sufficiently
generic constraints, one might expect this bound to be sharp, but it is
far from sharp for highly symmetrical configurations. For example, each
particle in the Coxeter--Todd kissing configuration is in contact with
$82$ others, so\break{} 
this configuration has far more contacts than the isostatic
condition requires, but it is nevertheless unjammed. The
$\Lambda_9$ kissing configuration also has plenty of contacts,
but it is not even locally jammed.

\section{Kissing configuration of the Barnes--Wall lattice}
\vspace{-0.03cm} 

In this section, we show that the kissing configuration of the
Barnes--Wall lattice is rigid.  This lattice is the densest sphere
packing known in $\R^{16}$, and it has the highest known kissing
number.

We will consider the $4320$ minimal vectors of $\Lambda_{16}$ as points
on the sphere of radius $2$, which is their usual normalization. Then
the inner products between distinct elements of the kissing
configuration lie in the set $\{\pm 2, \pm 1, 0, -4\}$.

Recall that the minimal vectors of $\Lambda_{16}$ are as
follows (see Conway and Sloane~\cite[page~129]{SPLAG}). There are $480$ of the form
$2^{1/2}(\pm e_i \pm e_j)$ with $i \ne j$, which span\break{} 
$2^{1/2}D_{16}$, and $3840$ of the form
$2^{-1/2}\big(\sum_{i\in I} \pm e_i\big)$, where $I$ is the
support of one of the $30$ codewords of weight $8$ in the
first-order Reed--Muller code $\mathcal{RM}(1,4)$ of block
length $16$ and where the number of minus signs is even.
Codewords in $\mathcal{RM}(1,4)$ correspond to affine linear
functions from $\F_2^4$ to $\F_2$, with the codeword consisting
of the values of the function at the $16$ points in $\F_2^4$.

\begin{proposition}
The kissing configuration of the Barnes--Wall lattice $\Lambda_{16}$ is
infinitesimally jammed.
\end{proposition}

\begin{proof}
Consider any infinitesimal deformation of this code.  The points
$2^{1/2}(\pm e_i \pm e_j)$ form a copy of the $D_{16}$ root system
(scaled by $2^{1/2}$ so the minimal distances match).  This root system
is infinitesimally jammed by \fullref{cor:rootjammed}, so after
applying an infinitesimal rotation we can assume that all these points
are fixed to first order.  Furthermore, each pair of points with inner
product $\pm 2$ spans a scaled copy of $A_2$, and each pair of
antipodal points is contained in such a hexagon, so these inner
products also do not change to first order.  Thus, only the inner
products $0$ and $\pm 1$ can possibly change.\pagebreak{} 

We will now show that the facts in the previous paragraph together
imply that the infinitesimal deformation must be identically zero. For
consider a minimal vector $x = 2^{-1/2}\sum_{i=1}^8 s_{i} e_{k_i}$,
where $1 \leq k_1 < \dots < k_8 \leq 16$ and $s_{i} \in \{\pm 1\}$, and
suppose $x$ is infinitesimally perturbed by $y = \sum_i y_i e_i$.

The inner product of $x$ with the $D_{16}$ minimal vector
$2^{1/2}(s_{i} e_{k_i} + s_{j} e_{k_j})$ is $2$ (note that the
supports overlap and the signs cancel), and as noted above this inner
product cannot change to first order.  Because all the $D_{16}$
vectors are fixed, we find that $s_{i}y_{k_i} + s_{j} y_{k_j} =
0$. This holds for every choice of $k_i$ and $k_j$, and if $k_i$,
$k_j$ and $k_\ell$ are distinct, then the equations
$$
s_{i}y_{k_i} + s_{j} y_{k_j} = s_{i}y_{k_i} + s_{\ell} y_{k_\ell}
= s_{j}y_{k_j} + s_{\ell} y_{k_\ell} = 0
$$
imply that $y_{k_i} = y_{k_j} = y_{k_\ell} = 0$.  Since this holds for
all $i = 1, \dots, 8$, we see that the nonzero coordinates of $x$
cannot change to first order.  Next we have to show that the zero
coordinates do not change either.

Let the positions of the zeros be $\ell_1 < \dots < \ell_8$,
and consider the $D_{16}$ vector $z = 2^{1/2}(e_{\ell_m} +
e_{\ell_n})$ for some $m \ne n$. It is orthogonal to $x$, and
by \fullref{lemma:orthoD4} below, $x$ and $z$ are contained in
a sublattice isometric to a scaled copy of $D_4$. Since $D_4$
is infinitesimally jammed, the inner product $\langle x, z
\rangle$ does not change to first order, and\break{} 
since $D_{16}$ is fixed, $z$ itself does not change to first
order. Hence $y_{\ell_m} + y_{\ell_n} = 0$. Similarly,
$y_{\ell_m} - y_{\ell_n} = 0$ by replacing $z$ with $z' =
2^{1/2}(e_{\ell_m} - e_{\ell_n})$. Thus, $y_{\ell_m} = 0$\break{} 
for all $m$, and $x$ does not change to first order, as
desired.
\end{proof}

All that remains is to show that every pair of orthogonal minimal
vectors in $\Lambda_{16}$ is contained in a scaled copy of $D_4$
(\fullref{lemma:orthoD4}). This would be trivial if the automorphism
group acted transitively on these pairs, because it is easy to exhibit
a pair for which it is true (from a copy of $D_4$ inside $D_{16}$, for
instance). However, there are actually two orbits. Each minimal vector
in $\Lambda_{16}$ is orthogonal to $1710$ others, and under the
stabilizer of that vector they form two orbits, of size $1680$ and
$30$.  This non-transitivity is actually quite remarkable, because it
corresponds to a decomposition of the $4320$ minimal vectors into
$135$ cross polytopes: to form one of the cross polytopes, take any
vector, its antipode and the stabilizer orbit of size $30$.  It is
far from obvious that this works, but it forms a decomposition that is
invariant under the automorphism group.  By contrast, the $240$
minimal vectors in $E_8$ form $15$ cross polytopes, and the $196560$
in $\Lambda_{24}$ form $4095$ (see, for example, Elkies~\cite[page~6,
footnote~3]{E}), but in these\break{} 
more symmetrical cases there is no invariant decomposition.

In principle, one could simply check \fullref{lemma:orthoD4}
for the two orbits, but we will give a direct proof that
requires less information about the action of the automorphism
group.\pagebreak{} 

\begin{lemma} \label{lemma:orthoD4}
Any two orthogonal vectors of $\Lambda_{16}$ are contained in a
sublattice isometric to $D_4$ scaled by $2^{1/2}$.
\end{lemma}

\enlargethispage{0.5cm} 

\begin{proof}
Let $x$ and $y$ be two orthogonal minimal vectors in $\Lambda_{16}$.
Without loss of generality, we may assume that $x = 2^{1/2}(e_1+e_2)$,
because the automorphism group acts transitively on the minimal vectors
(see Griess~\cite[Theorem~10.2]{G} and~\cite{GCorr}).  We will construct
vectors $w$ and $z$ such that $w,x,y,z$ form a $D_4$ Dynkin diagram, up
to scaling: 
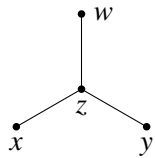
\begin{center}
\begin{tikzpicture}
\fill (0,0) circle (0.05);
\draw (0,0) node[below] {$z$};
\fill (-0.866025403,-0.5) circle (0.05);
\draw (-0.866025403,-0.5) node[below] {$x$};
\fill (0.866025403,-0.5) circle (0.05);
\draw (0.866025403,-0.5) node[below] {$y$};
\fill (0,1) circle (0.05);
\draw (0,1) node[right] {$w$};
\draw (0,0) -- (-0.866025403,-0.5);
\draw (0,0) -- (0.866025403,-0.5);
\draw (0,0) -- (0,1);
\end{tikzpicture}
\end{center}
In other words, $\langle w,x \rangle = \langle x,y \rangle = \langle
y,w \rangle = 0$ and $\langle w,z \rangle = \langle x,z \rangle =
\langle y,z \rangle = -2$.\break{} 
Then these vectors will span a copy of $2^{1/2} D_4$.

First, we analyze the case when $y \in 2^{1/2} D_{16}$.  Then either $y
= \pm 2^{1/2} (e_1-e_2)$ or $y = 2^{1/2}(\pm e_i \pm e_j)$ with $2 < i
< j$.  If $y = \pm 2^{1/2} (e_1-e_2)$, then we can assume the $\pm$
sign is positive (by replacing $y$ with $-y$ if needed, since this does
not change whether $x$ and $y$ are contained in a scaled copy of
$D_4$), and we can take $z = 2^{1/2}(e_3-e_1)$\break{} 
and $w = 2^{1/2}(e_4-e_3)$.  If $y = 2^{1/2}(\pm e_i \pm e_j)$
with $2 < i < j$, then we can assume\break{} 
that $y = 2^{1/2}(e_i \pm e_j)$ and take $z =
-2^{1/2}(e_1+e_i)$ and $w = 2^{1/2}(e_1-e_2)$.

In the remaining case, we have $y = 2^{-1/2} \sum_{i=1}^8 s_i e_{k_i}$
with $s_i \in \{\pm1\}$ and $k_1 < k_2 < \dots < k_8$.  Since $x$ is
orthogonal to $y$, either $k_1=1$ and $k_2=2$, or $k_1>2$.  If $k_1=1$
and $k_2=2$, then $s_1+s_2=0$, and without loss of generality we can
take $s_1=1$ and $s_2=-1$; then let $w = 2^{1/2}(s_3 e_{k_3} - s_4
e_{k_4})$ and $z = -2^{1/2}(e_1 + s_3 e_{k_3})$.

Finally, suppose $k_1>2$.  This case is slightly more subtle,
because $z$ cannot be an element of $2^{1/2} D_{16}$, so we
must instead produce a Reed--Muller codeword.  The code
$\mathcal{RM}(1,4)$ is five-dimensional, so we can find a
nonzero codeword satisfying any four linear conditions.
Specifically, we require it to vanish in coordinates $1$, $2$,
$k_1$ and $k_2$.  The resulting codeword $c$ has weight $8$,
as do all nonconstant codewords.

The complementary codeword
$(\hspace{-0.08em}1\hspace{-0.08em},\hspace{-0.08em}\dots\hspace{-0.08em},\hspace{-0.08em}1\hspace{-0.08em}){-}c$
is also in
$\mathcal{RM}(\hspace{-0.08em}1\hspace{-0.08em},\hspace{-0.08em}4\hspace{-0.08em})$.
Its support
$\{\hspace{-0.08em}\ell_1\hspace{-0.08em},\hspace{-0.08em}\dots\hspace{-0.08em},\hspace{-0.08em}\ell_8\hspace{-0.08em}\}$
contains $1$, $2$, $k_1$ and $k_2$, and it must intersect
$\{k_1,\dots,k_8\}$ in exactly four elements.

Suppose $\ell_1=1$, $\ell_2=2$, and $\ell_3$ and $\ell_4$ are
the other two elements of $\{\ell_1,\dots,\ell_8\} \setminus
\{k_1,\dots,k_8\}$. Now let $z = 2^{-1/2}\sum_{i=1}^8 t_i
e_{\ell_i}$ with $t_1=t_2=-1$ and $t_i = -s_j$ whenever $\ell_i
= k_j$.  (The signs $t_3$ and $t_4$ can be chosen arbitrarily,
subject to $z$ having a even number of minus signs.) Finally,
let $w = -2^{1/2}(t_3 e_{\ell_3} + t_4 e_{\ell_4})$. As in the
previous cases, the vectors $w,x,y,z$ span a copy of $2^{1/2}
D_4$, which completes the proof.
\end{proof}
\pagebreak{} 

\section{Kissing numbers in dimensions $25$ through $31$}
\label{section:kissing25to31} \vspace{-.2cm} 

Before this paper, the best kissing numbers known in dimensions
$25$ through $31$ were those shown in
\fullref{table:previouskiss} (achieved by laminated lattices,
see Conway and Sloane~\cite{Laminated}, with a correction from
Mus\`es~\cite{Muses}).  These numbers increase surprisingly
slowly from each dimension to the next, which made it difficult
to believe they could be optimal, but\break{} 
they had not been improved since 1982, despite improvements
starting in $32$ dimensions from Edel, Rains and
Sloane~\cite{ERS}.

\begin{table}[ht!]
\centering
\begin{tabular}{cc}
Dimension & Previous record \\ \hline
25 & 196656 \\
26 & 196848 \\
27 & 197142 \\
28 & 197736 \\
29 & 198506 \\
30 & 200046 \\
31 & 202692
\end{tabular}
\vspace{-0.2cm} 
\caption{Previous record kissing numbers}
\label{table:previouskiss}
\vspace{-0.2cm} 
\end{table}

All of these codes were based on the Leech lattice
$\Lambda_{24}$ via a recursive construction in which each
contains the previous one as a cross section.  Recall that
$\Lambda_{24}$ is the unique densest $24$--dimensional lattice
(see Cohn and Kumar~\cite{Leechopt}) and its kissing
configuration of $196560$ minimal vectors of norm $4$ is the
unique optimal $24$--dimensional kissing arrangement (see
Odlyzko and Sloane~\cite{OS}, Leven{\v{s}}te{\u\i}n~\cite{L}
and Bannai and Sloane~\cite{BS}).

The old $25$--dimensional kissing number is just $96$ more than
the value achieved in\break{} 
$\R^{24}$, and it is obtained by putting cross polytopes above
and below the Leech lattice\break{} 
kissing configuration. Specifically, in the coordinates used by
Conway and Sloane \cite[\break
page~133]{SPLAG}, the Leech lattice contains a
copy of $2^{1/2}D_{24}$, and to extend its kissing
configuration to $25$ dimensions one simply includes the
remaining minimal vectors in\break{} 
$2^{1/2}D_{25}$. Much like the $\Lambda_9$ kissing
configuration in $\R^9$ (which is constructed in a precisely
analogous way from $E_8$), this code is not even locally
jammed: all the points not in $2^{1/2}D_{25}$ can move.
However, the analogy with $9$ dimensions is not perfect, and
the $25$--dimensional version of \fullref{prop:optikiss} is not
even true. We begin this section by showing how to improve the
$25$--dimensional kissing number while retaining the Leech
lattice kissing configuration as a cross section.  Once we have
shown that, we will abandon the cross section constraint and go
on to improve all the kissing numbers from $25$ through $31$
dimensions.

Given a vector $v \in \Lambda_{24}$ of norm $6$, consider the
vectors $(2v/3,\pm2/\sqrt{3}) \in \R^{25}$, which lie on the
same sphere (of radius $2$) as the minimal vectors of the Leech
lattice\pagebreak{} 
$\Lambda_{24} \times \{0\}$.  If $x$ is any minimal vector in
the Leech lattice, then $|v-x|^2 \ge 4$, from which it follows
that $\langle v,x \rangle \le 3$. Thus,
$$
\big\langle (2v/3,\pm2/\sqrt{3}),(x,0) \big\rangle \le 2,
$$
so all the new vectors stay far enough away from those in the Leech
lattice.  However, they may come too close to each other if we use
vectors $v$ that are too close.  To avoid that problem, we choose a
subset $R$ of the vectors of norm $6$ in $\Lambda_{24}$ such that for
$v,w\in R$ with $v \ne w$, we have $\langle v,w \rangle \le 1$.  Then
$$
\big\langle (2v/3,\pm2/\sqrt{3}),(2w/3,\pm2/\sqrt{3}) \big\rangle
\le 4/9 + 4/3 < 2,
$$
as desired.  Finally,
$$
\big\langle (2v/3,2/\sqrt{3}),(2v/3,-2/\sqrt{3})\big\rangle
= 4\cdot6/9 - 4/3 < 2.
$$
Thus,
$$
\{(x,0) : x \in \Lambda_{24},\, |x|^2=4\} \cup 
\{(2v/3,\pm2/\sqrt{3}) : v \in R\}
$$
achieves a kissing number of $196560+2|R|$ in $\R^{25}$.

We do not know how to maximize the size of $R$, but greedy
approaches easily yield $|R|>48$ via a computer search and thus
improve the kissing number. The best we have achieved is
$|R|=70$, and the vectors are listed in the supplementary
infor\-mation 
(see \fullref{appendix:data}).  If we take into account that
all the inner products are in\break{} 
$\{-6, -4,-3,-2,\pm 1, 0\}$, then linear programming bounds
(see Delsarte, Goethals and Seidel~\cite{DGS}) prove that $|R|
\le 280$. Thus, our construction of $R$ cannot be improved by
more than a factor of four, but we do not know how close to
optimal it is.

We do not have a conceptual explanation of the $70$--point configuration
produced by the computer search, but we do have a simple construction
that achieves $|R|=56$, which is enough to improve the kissing number
in $\R^{25}$.

\begin{proposition}\label{lemma:56}
There is an antipodal set $R$ of vectors of norm $6$ in $\Lambda_{24}$
such that $|R|=56$ and $|\langle x, y \rangle| \le 1$ for all distinct,
non-antipodal $x, y \in R$.
\end{proposition}

\begin{proof}
The Leech lattice $\Lambda_{24}$ contains two orthogonal copies
of the Coxeter--Todd lattice $K_{12}$ (see Conway and
Sloane~\cite[page~128, equations (129) and (130)]{SPLAG}).
Thus, it suffices to find such a subset of size $28$ within
$K_{12}$.  To do so, we will use the construction of $K_{12}$
given in Section~2.5 of Conway and Sloane~\cite[Section~2.5,
page~426]{CS83}, which is called $\Lambda^{(7)}$ in that paper.

This construction realizes $K_{12}$ as an Eisenstein lattice,
that is, a $\Z[\omega]$--module with $\omega = e^{2\pi i/3}$.
Let $\alpha = 2+3\omega$, so $|\alpha|^2 = 7$.  Then the
Coxeter--Todd lattice $K_{12}$ is\pagebreak{} 
isometric to $L$, where
$$L = \frac{1}{\alpha} \bigg\{ (x_1,\dots,x_7) \in \Z[\omega]^7 : x_1
  \equiv x_2 \equiv \dots \equiv x_7 \pmod{\alpha} \text{ and }
  \sum_{j=1}^7 x_j = 0 \bigg\}.$$
As before, we use the inner product $\langle x,y \rangle = 2\Re
\langle x,y \rangle_\C$, where $\langle z,w \rangle_\C = \sum_{j=1}^7
z_j \wwbar{w}_j$ is the usual Hermitian inner product on $\C^7$.

We will make use of only some of the vectors of norm
$6$ in $L$, namely those obtained by permuting the coordinates of
$(1,\omega,\omega^2,0,0,0,0)$.
Let $W$ be the set of such points.  Note that $1+\omega+\omega^2=0$ and
$\alpha \equiv \alpha \omega \equiv \alpha \omega^2 \equiv 0
\pmod{\alpha}$, so these vectors are indeed in $L$.

We begin by finding a six-dimensional regular simplex in
$K_{12}$ (that is, seven vectors\break{} 
of norm $6$ whose inner products are all $-1$). To do so, start
by choosing seven vectors in $W$ such that the supports of any
two overlap in exactly one coordinate position.  Equivalently,
label the seven positions by points in the Fano plane
$\Proj^2(\F_2)$, and choose the seven lines in that plane as
the supports of the vectors.

Suppose two of these vectors overlap in position $i$, with coordinates
$\omega^j$ and $\omega^k$.  Then their inner product in $L$ is $2\Re \, 
\omega^{j-k}$, which equals $-1$ if and only if $j \ne k$.  Thus,\break{} 
we must ensure that each pair of vectors differ in the position
where they overlap.

Equivalently, labeling the incident point-line pairs in $\Proj^2(\F_2)$
with $1$, $\omega$ or $\omega^2$ must yield a three-coloring of the
edges of the bipartite incidence graph between the points and lines
(the Heawood graph).  The existence of such a coloring follows from
K\H{o}nig's line coloring theorem (see Lov\'asz and
Plummer~\cite[Theorem~1.4.18]{LP09}), which says that the chromatic
index of a bipartite graph equals its maximal degree. It is not difficult
to exhibit such a coloring, and hence a six-dimensional simplex in
$K_{12}$, explicitly as the rows of the following array:
$$
\begin{array}{ccccccc}
0 & 1 & 0 & \omega & 0 & \omega^2 & 0\\
1 & 0 & 0 & \omega^2 & \omega & 0 & 0\\
0 & 0 & \omega & 1 & 0 & 0 & \omega^2\\
\omega & \omega^2 & 1 & 0 & 0 & 0 & 0\\
0 & \omega & 0 & 0 & \omega^2 & 0 & 1\\
\omega^2 & 0 & 0 & 0 & 0 & 1 & \omega\\
0 & 0 & \omega^2 & 0 & 1 & \omega & 0
\end{array}
$$
To produce $14$ of the desired $28$ points in $L$, we simply
take these seven points and their antipodes.  The remaining
$14$ can be obtained from them by permuting the
coordinates.\break{}\pagebreak{}

Specifically, we fix the seventh coordinate and permute the
others via the six-cycle
$$
1 \mapsto 6 \mapsto 4 \mapsto 3 \mapsto 2 \mapsto 5 \mapsto 1
$$
or its inverse.  It is not difficult to check that if we apply either
permutation to the original simplex in $L$, then the seven new points
all have inner products $0$ or $\pm1$ with the original points.

Thus, we have found $28$ points in $K_{12}$ with the desired inner
products, and hence $56$ in $\Lambda_{24}$.
\end{proof}

In the remainder of this section, we will further improve the records
for kissing numbers in dimensions $25$ through $31$.  All of our
constructions will be quite simple, except that they rely on finding a
subset $S$ of the minimal vectors in $\Lambda_{24}$ such that no two
elements of $S$ have inner product greater than $1$.  We will give a
conceptual argument that achieves $|S|=288$, which will be enough to
improve all the kissing numbers.  Further optimization using a computer
has led to an antipodal subset with $|S|=480$, which yields the new
records shown in \fullref{table:kissing}.  The code of size $480$ is
enumerated in\break{} 
the accompanying computer files (see \fullref{appendix:data}).
If we take into account that all the inner products are in
$\{-4, -2, \pm 1, 0\}$, then linear programming bounds prove
that $|S| \le \lfloor 9360/11\rfloor = 850$.

\begin{lemma}\label{lemma:288}
There is an antipodal set $S$ of minimal vectors in $\Lambda_{24}$ such
that $|S|=288$ and $|\langle x, y \rangle| \le 1$ for all distinct,
non-antipodal $x, y \in S$.
\end{lemma}

The code we construct below is relatively well known, but for
completeness we sketch the construction here.

\begin{proof}
Let $\mathcal{NR}$ be the Nordstrom--Robinson code, a nonlinear code of
size $256$ and distance $6$ in $\F_2^{16}$ (see MacWilliams and
Sloane~\cite[page~73]{MS}). Let
$W$ consist of all the vectors $\big(\sum_i (-1)^{s_i} e_i\big)/2$ with
$(s_1, \dots, s_{16}) \in \mathcal{NR}$, as well as the vectors $\pm
2e_i$\break{} 
for $1 \leq i \leq 16$. The elements of $W$ have norm $4$, and
is easy to check that distinct vectors of $W$ have inner
product at most $1$ with each other. One can also check that\break{} 
the span of $W$ is isomorphic to the Barnes--Wall lattice.
Since there is an embedding\break{} 
of the Barnes--Wall lattice into the Leech lattice (see Conway
and Sloane~\cite{Laminated}) we obtain the desired subset $S$
of $\Lambda_{24}$ as the image of $W$.
\end{proof}

For the rest of this section, $S$ will denote any set of minimal
vectors in $\Lambda_{24}$ such that $\langle x, y \rangle \le 1$ for
all distinct $x, y \in S$.  In particular, it may be the set from
\fullref{lemma:288},\break{} 
the one of size $480$ found by a computer search, or an even
larger set.  Also, let $\sC$ denote the set of minimal vectors
in $\Lambda_{24}$.

\begin{proposition} \label{prop:25}
The kissing number in dimension $25$ is at least $196560 + |S|$.
\end{proposition}

\begin{proof}
Fix an angle $\theta$ (to be determined shortly), and let
$$
T = \{(x,0) : x \in \sC \setminus S \}
\cup \{(x \cos \theta,\pm 2 \sin \theta) : x \in S \}.
$$
In other words, we remove $S$ from the equatorial hyperplane and
replace it with two copies lying on the parallels at latitude $\pm
\theta$.

The inner product between a vector $(x,0)$ and $(y \cos\theta , \pm 2
\sin \theta)$ is $\langle x,y \rangle \cos \theta \le 2$, since $x \ne
y$ (and hence $\langle x,y \rangle \le 2$).

The inner product between two distinct vectors in the northern
hemisphere is
\begin{align*}
\big\langle (x \cos \theta, 2 \sin \theta), (y \cos \theta,2 \sin \theta) \big\rangle
&= \langle x, y \rangle \cos^2 \theta + 4 \sin^2\theta \\
& \leq \cos^2\theta  + 4\sin^2\theta\\
& = 1 + 3 \sin^2 \theta,
\end{align*}
which is at most $2$ provided that $\sin^2 \theta \le 1/3$.

The inner product between two vectors in opposite hemispheres is
\begin{align*}
\big\langle (x \cos \theta, 2 \sin \theta), (y \cos \theta,-2\sin \theta) \big\rangle
&= \langle x, y \rangle \cos^2\theta  - 4\sin^2\theta \\
& \leq 4\cos^2\theta  - 4\sin^2\theta\\
& \le 4 - 8 \sin^2\theta,
\end{align*}
which is at most $2$ provided that $\sin^2\theta \ge 1/4$. Therefore, for any
choice of $\theta$ such that $1/4 \le \sin^2 \theta \le 1/3$, the code
$T$ of size $196560 + |S|$ has maximal inner product $2$ and is thus a
kissing configuration.
\end{proof}

To deal with higher dimensions, we will use disjoint sets
$S_1,\dots,S_{51} \subset \sC$ with the same property that $\langle x,
y \rangle \le 1$ for all distinct $x, y \in S_i$.

\begin{lemma} \label{lemma:probmethod}
Let $S_1=S$.  Then there are disjoint subsets $S_2,S_3,\ldots \subset
\sC$ such that $\langle x, y \rangle \le 1$ for all distinct $x, y \in
S_i$ and
$$
|S_i| \ge |S|\bigg(1 - \frac{\sum_{j=1}^{i-1}|S_j|}{196560} \bigg).
$$
Furthermore, if $S$ is antipodal, then we can take $S_2,S_3,\dots$ to
be antipodal as well.
\end{lemma}

\pagebreak 

Of course, $S_i = \emptyset$ if $i$ is sufficiently large. The proof of
\fullref{lemma:probmethod} uses the\break{} 
probabilistic method, rather than an explicit construction.
This approach is surely not\break{} 
optimal, but it is convenient. Furthermore, it naturally leads
to a probabilistic algorithm that can be used to obtain these
sets with the aid of a computer.

\begin{proof}
Let $G$ be the automorphism group of the Leech lattice.  We will obtain
the sets $S_2,S_3,\dots$ by induction as subsets of translates of $S$
by elements of $G$.  First, we compute the expected size of the
intersection of $gS$ with $S_1 \cup \dots \cup S_{i-1}$ when $g$ is
chosen uniformly at random from $G$.  Since $S_1,\dots,S_{i-1}$ are
disjoint, we have
$$
\bigexpt_{g \in G} \big|gS \cap (S_1 \cup S_2 \dots \cup S_{i-1})\big| =
\sum_{j=1}^{i-1} \bigexpt_{g \in G} \big|gS \cap S_j\big|.
$$
The terms on the right side are easily computed, via
\begin{align*}
\bigexpt_{g \in G}\big|gS \cap S_j\big|
&=
\sum_{x \in S} \,\Prob_{g \in G} [gx \in S_j] \\ 
&= \sum_{x \in S} \sum_{y \in S_j} \Prob_{g \in G}[gx = y] \\
&= \frac{|S||S_j|}{196560}.
\end{align*}
In the last line, we have used the fact that $G$ acts transitively on
the $196560$ minimal vectors; therefore, $gx$ is uniformly distributed
among those vectors as $g$ varies, and
$$
\Prob_{g \in G}[gx = y] = \frac{1}{196560}.
$$

Now there must be a translate $gS$ whose overlap with $S_1 \cup \dots
\cup S_{i-1}$ is at most the expectation, namely
$$
|S|\frac{\sum_{j=1}^{i-1}|S_j|}{196560}.
$$
Letting $S_i = gS \setminus \bigcup_{j=1}^{i-1} S_j$ completes the
proof.  (Note that if $S$ is antipodal, then\break{} 
each of these sets will be as well.)
\end{proof}

\begin{theorem} \label{theorem:master}
Let $S_1,S_2,\dots$ be as in \fullref{lemma:probmethod}.  Suppose we
have partitioned a kissing configuration in the unit sphere $S^{d-1}
\subset \R^d$ into subsets $T_1,\dots,T_N$ such that for distinct $x,y
\in T_i$, we have $\langle x,y \rangle \le -1/2$.  Then the kissing
number in dimension $24+d$ is at least
$$
\sum_{i=1}^N \big(|T_i|-1\big) |S_i|,
$$
and it is achieved by the set
$$
\bigg\{(x,0) \in \R^{24} \times \R^d :
x \in \sC \setminus \bigcup_{i=1}^N S_i \bigg\} \cup \bigcup_{i=1}^N
\big\{\big(x \sqrt{2/3}, y \sqrt{4/3}\big) : x \in S_i, y \in T_i\big\}.
$$
\end{theorem}

Note that the minimal angle constraint on $T_i$ implies that $|T_i| \le
3$, and $|T_i|=3$ is possible only if $T_i$ consists of three points
forming an equilateral triangle on a great circle. Of course it is
pointless to take $|T_i|=1$, and to optimize the bound we should take
$|T_1| \ge |T_2| \ge \dots \ge |T_N|$, assuming $|S_1| \ge |S_2| \ge
\dots \ge |S_N|$.

\begin{proof}
First, observe that the points $\big(x \sqrt{2/3}, y \sqrt{4/3}\big)$ have
norm
$$
\tfrac{2}{3}|x|^2  + \tfrac{4}{3}|y|^2 =
\tfrac{2}{3} \cdot 4 + \tfrac{4 }{3}\cdot 1 = 4,
$$
as they should.  Within each set $\big\{\big(x \sqrt{2/3}, y
\sqrt{4/3}\big) : x \in S_i, y \in T_i\big\}$, the inner product
between two distinct points $\big(x \sqrt{2/3}, y \sqrt{4/3}\big)$ and
$\big(x' \sqrt{2/3}, y' \sqrt{4/3}\big)$ is at most
$$
\tfrac{2}{3} \cdot 4 - \tfrac{4}{3} \cdot \tfrac{1}{2} = 2
$$
if $y \ne y'$, and at most
$$
\tfrac{2}{3} \cdot 1 + \tfrac{4}{3} \cdot 1 = 2
$$
if $x \ne x'$.  Inner products between two of these sets are also
easily dealt with: for $i \ne j$, the inner product between $\big(x
\sqrt{2/3}, y \sqrt{4/3}\big)$ with $x \in S_i$, $y \in T_i$ and
$\big(x' \sqrt{2/3}, y' \sqrt{4/3}\big)$ with $x' \in S_j$, $y' \in
T_j$ is at most
$$
\tfrac{2}{3} \cdot 2 + \tfrac{4}{3} \cdot \tfrac{1}{2} = 2.
$$
Finally, the inner product between $\big(x \sqrt{2/3}, y
\sqrt{4/3}\big)$ and $(x',0)$ is at most $2 \sqrt{2/3}$, which is
strictly less than $2$.
\end{proof}

Note that the code from \fullref{theorem:master} cannot be even
locally jammed, provided that $\sC \ne \bigcup_{i=1}^N S_i$, because
the points in
$$
\bigg\{(x,0) \in \R^{24} \times \R^d : x \in \sC \setminus \bigcup_{i=1}^N
S_i \bigg\}
$$
are not adjacent to any of the points of the code that are not in this set.

We do not know the best way to apply \fullref{theorem:master}, but
we can use it as follows.  Given\break{} 
any kissing configuration in $\R^d$ of size $K$, we want to
partition it into antipodal pairs and equilateral triangles (or
singletons if necessary).  Of course, if it is antipodal we can
simply partition it into $K/2$ antipodal pairs, but that will
generally not be optimal.\pagebreak{}\break{} 
If there is an Eisenstein structure,
as is the case for $A_2$, $D_4$, $E_6$ and $E_8$, then we
partition it into regular hexagons using that structure and
divide each hexagon into two equilateral triangles. For $A_3$,
$D_5$ and $E_7$, we partition a cross section using its
Eisenstein structure and then fill the rest with antipodal
pairs. This yields the bounds shown in \fullref{table:bounds}.
To compute $|S_i|$, we simply use \fullref{lemma:probmethod}
recursively, taking into account that $|S_i|$ must be not just
an integer, but also even when $S_i$ is antipodal. For example,
if $|S|=480$ with $S$ antipodal, then $|S_2| \ge 480 \cdot
(1-480/196560) \approx 478.83$, which implies $|S_2| = 480$.

\renewcommand{\arraystretch}{1.2} 

\begin{table}
\centering
\begin{tabular}{cccc}
Dimension & Lower bound & $|S|=288$ & $|S|=480$\\ \hline
$25$ & $196560 + |S_1|$ & 196848 & 197040\\
$26$ & $196560 + 2|S_1|+2|S_2|$ & 197712 & 198480\\
$27$ & $196560 + 2|S_1| + 2|S_2| + \sum_{i=3}^5 |S_i|$ & 198576 & 199912 \\
$28$ & $196560 + 2\sum_{i=1}^8 |S_i|$ & 201156 & 204188 \\
$29$ & $196560 + 2\sum_{i=1}^8 |S_i| + \sum_{i=9}^{16} |S_i|$ & 203430 & 207930 \\
$30$ & $196560 + 2\sum_{i=1}^{24} |S_i|$ & 210200 & 219008 \\
$31$ & $196560 + 2\sum_{i=1}^{24} |S_i| + \sum_{i=25}^{51} |S_i|$ & 217588 & 230872
\end{tabular}
\caption{New lower bounds on kissing numbers} \label{table:bounds}
\end{table}

%

\renewcommand{\arraystretch}{1} 

We do not expect that these bounds are anywhere close to being optimal,
and already in $\R^{32}$ they are worse than the previous record: they
give $266544$, while the record is $276032$ ($|S|=554$ would be
required to break the record).  However, our bounds improve upon all
cases from dimension $25$ to $31$.

\vspace{-0.15cm} 
\section{Open problems}
\vspace{-0.15cm} 

We conclude with some open problems.  The most basic is whether there
are spherical codes that span the ambient space and are jammed but not
infinitesimally jammed.  We suspect that there are such codes, but we have
not found one.  A related question is how to test efficiently whether a
code is jammed (which would be possible if jamming and infinitesimal
jamming were equivalent, although one could hope for even faster
algorithms). Such an algorithm might be based on higher-order variants
of infinitesimal jamming, but they are far more subtle than one might
expect (see Connelly and Servatius~\cite{ConS}).

Our new kissing records in dimensions $25$ through $31$ can presumably
be improved, and it would be very interesting to know how far these
techniques can be pushed, or how\pagebreak{}\break{} 
to construct much better arrangements. The ratio of the upper
bounds from Shtrom~\cite{S} to the lower bounds proved here
grows roughly like $1.4^{d}$ in $24+d$ dimensions for $1 \le d
\le 7$, so there is considerable room for improvement.  Note
that optimizing the sizes of the sets $R$ and $S$ from
\fullref{section:kissing25to31} can be viewed as maximizing the
size of cliques in highly symmetrical graphs.  Finding large
cliques is NP-hard in general, but that does not settle the
question of how well one can solve this problem in practice.

Among kissing problems in low dimensions, dimensions $17$ through $23$
seem ripe for improvement, although it is unlikely that the approach
we have used here can be made to work.  Dimension $16$ is particularly
interesting, and we would very much like to know whether the $4320$
minimal vectors of the Barnes--Wall lattice solve the kissing problem.
They are certainly not the unique solution, because we have found
that at least one of the $16$--dimensional packings constructed by Conway
and Sloane~\cite{Allthebest} has a different kissing configuration
of the same size.  It seems unlikely that one could classify all the
possibilities in $\R^{16}$, but it might be possible to extend the
conjecturally exhaustive list from \fullref{section:kisslist} to higher
dimensions, perhaps up through $\R^{12}$.

In \fullref{section:9d} we gave a list of the best kissing
configurations known in $\R^9$ through\break{} 
$\R^{12}$, but we suspect even more such configurations remain
to be found.  As mentioned in that section, the $E_7$ root
system fits into the same framework.  In fact its three other
competitors do too, in the following sense. In each case, one
can find a cross polytope contained in the configuration (for
example, by a randomized, greedy algorithm).  Using it to
define the coordinate system yields simple rational coordinates
for each point. In the $E_7$ case, these coordinates come from
a constant weight code. In the other three cases they do not,
but the coordinates can be obtained by a systematic mutation of
the $E_7$ case. Perhaps one could obtain additional kissing
configurations in nine through twelve dimensions via a similar
construction.

The problem of rigidity naturally generalizes to many other ambient
spaces, such as projective spaces, but this generalization presents new
and interesting features.  For example, a projective configuration is
not determined by its pairwise distance matrix, and in complex
projective space there are even continuous families of optimal codes
with exactly the same pairwise distances (see Cohn and Kumar~\cite[page~129]{Univopt}).

Even in Euclidean space, which is the most thoroughly studied case so
far, there are unresolved questions.  How can one test whether a
periodic packing is jammed?  If one restricts attention to packings
consisting of $N$ translates of a lattice (that is, those with $N$ particles
per unit cell), then one can test for infinitesimal jamming
and hence jamming (see Donev, Torquato, Stillinger and
Connelly~\cite{DTSC}).  However, the answer may depend on $N$.
For example, the laminated lattice $\Lambda_9$ is jammed with $N=1$
(that is, it cannot be deformed as a lattice packing), but not with $N=2$
(giving rise to the\pagebreak{}\break{} 
fluid diamond packings of Conway and Sloane~\cite{Allthebest}).
We know of no bound for how large $N$ must be to detect a lack
of rigidity.

Finally, we conjecture that in all sufficiently high
dimensions, there exist optimal kissing configurations with no
contacts whatsoever (that is, no pairs of points with inner
product $1/2$), so they are unjammed in the strongest possible
sense.  This phenomenon occurs in $\R^3$, but we know of no
higher-dimensional cases.  Part of our motivation for making
this conjecture is that we know of no large, jammed kissing
configurations at all in high dimensions.  The $D_n$ root
system is jammed for $n>3$, but it contains only $2n(n-1)$
points, which is tiny compared with the exponential growth of
the kissing number (see, for example,
Conway and Sloane~\cite[pages~23--24]{SPLAG}).  Are there exponentially large
jammed kissing configurations in high dimensions?  Are there
even any of greater than quadratic size?  For example, is the
kissing configuration of the Barnes--Wall lattice $BW_{2^k}$ in
$\R^{2^k}$ always jammed?  (It has size asymptotic to $C \cdot
2^{k(k+1)/2}$, where $C$ is a constant \cite[page~24]{SPLAG}.)

\vspace{-0.09cm} 
\section*{Acknowledgements}
\vspace{-0.09cm} 

We thank K\'aroly Bezdek, Noam Elkies, Michel Goemans and Achill
Sch\"urmann for\break{} 
helpful discussions, Patric \"Osterg\r{a}rd for providing
computer files for Best's constant\break{} 
weight codes, and an anonymous referee for providing useful
feedback on the manu\-script. 
Abhinav Kumar was supported in part by National Science
Foundation grants DMS-0757765 and DMS-0952486 and by a grant
from the Solomon Buchsbaum Research Fund, and he thanks
Princeton University for its hospitality. Yang Jiao\break{} 
and Salvatore Torquato were supported in part by NSF grants
DMS-0804431 and DMR-0820341.

\appendix

\vspace{-0.09cm} 
\section{Data files}
\label{appendix:data} \vspace{-0.09cm} 

As supplementary information for this paper, we
have made available sixteen data files through the
\texttt{arXiv.org} e-print archive, where it is paper number
\href{http://arxiv.org/abs/1102.5060}{\texttt{arXiv:1102.5060}}.
One can access these files by downloading the source files for the paper.
The data files can also be downloaded from the web page for this article
(\href{http://dx.doi.org/10.2140/gt.2011.15.2235}{doi:10.2140/gt.2011.15.2235}).

Ten of them describe the kissing configurations enumerated in
\fullref{section:kisslist}.  These files are each named after the
corresponding configuration: \texttt{5-40a.txt}, \texttt{5-40b.txt},
\texttt{6-72a.txt} through \texttt{6-72d.txt}, and \texttt{7-126a.txt}
through \texttt{7-126d.txt}.  The first line of the file specifies the
number $N$ of points, and the second line specifies the dimension $n$
of the ambient Euclidean space.  The third line consists of $n$
positive\pagebreak{}\break{} 
integers $d_1,\dots,d_n$ (separated by spaces), which are the
coefficients of the diagonal quadratic form used to measure
distances. Finally, the remaining $N$ lines each give the $n$
coordinates of one of the points in the configuration (again
separated by spaces).  The scaling has been chosen so that all
the coordinates will be integers, and the inner product between
points $x$ and $y$ is defined by
$$
\langle x,y \rangle = \sum_{i=1}^n d_i x_i y_i.
$$
Equivalently, if we use the standard inner product, then the
$i$th coordinate must be scaled by $\sqrt{d_i}$, but phrasing
it in terms of changing the inner product avoids the need to
use irrational numbers as coordinates.  (There exist different
coordinate systems that use only rational numbers, even with
the standard inner product, but the coordinates used here are
compatible with the fibering construction from
\fullref{section:kisslist}.) Note that within each file,
all the vectors have the same norm, but they are not unit
vectors.

Four of the files, namely \texttt{b9-18.txt},
\texttt{b10-30.txt}, \texttt{b11-35.txt} and
\texttt{b12-51.txt}, describe the constant weight binary codes
used to build the kissing configurations from \fullref{9to12}
in \fullref{section:9d}. The file \texttt{bn-N.txt}
contains all the codes of block length $n$, size $N$, constant
weight $4$ and minimal distance $4$, up to isomorphism. (We
thank Patric \"Osterg\r{a}rd for providing these codes.)  Each
codeword is given on a line by itself, with the binary digits
separated by spaces, and there is a blank line between
different codes.

The remaining two files, \texttt{R.txt} and \texttt{S.txt},
describe the $70$--point configuration $R$ and the $480$--point
configuration $S$ from \fullref{section:kissing25to31}.
These files are in a slightly different format: they omit the
first three lines ($N$, $n$ and $d_1,\dots,d_n$).  Instead,
each line\break{} 
specifies the $24$ coordinates of one of the points in the
coordinate system used by Conway and Sloane~\cite[page~133,
Figure~4.12]{SPLAG}, but with the irrational factor of
$1/\sqrt{8}$ omitted.

\vspace{-0.2cm} 
\providecommand{\BIBZs}{Zs}

\end{document}